\documentclass[twoside]{article}

\usepackage[accepted]{aistats2020}
\usepackage[round]{natbib}

\usepackage{amsmath}
\usepackage{graphicx,psfrag,epsf}
\usepackage{enumerate}
\usepackage{float}
\usepackage{algorithm}
\usepackage[noend]{algpseudocode}

\usepackage{multirow}
\usepackage{lipsum}
\usepackage{amsfonts}
\usepackage{cite}      
                  
\usepackage{booktabs}
\usepackage{enumitem}

\usepackage{amsthm}
\usepackage{amssymb}
\usepackage{subfigure}
\usepackage{color}

\usepackage{wrapfig}

\newcommand{\bm}[1]{\boldsymbol{#1}}
\newcommand{\E}{\mathbb{E}}

\renewcommand{\P}{\mathbb{P}}

\newcommand{\Y}{\bm{Y}}

\newcommand{\mE}{\mathcal{E}}

\newcommand{\mB}{\mathcal{B}}

\newcommand{\wt}[1]{\widetilde{#1}}

\newcommand{\F}{\mathcal{F}}

\newcommand{\Var}{\mathsf{Var\,}}
\newcommand{\1}{\mathbb{I}}

\newcommand{\expect}[1]{\langle\mbox{$ #1$}\rangle}

\newcommand{\floor}[1]{\lfloor{#1}\rfloor}

\newcommand{\st}[1]{{\color{blue}#1}}

\def\B{\bm{B}}
\def\b{\bm{\beta}}
\def\Y{\bm{Y}}

\def\x{\bm{x}}

\def\mV{\mathcal{V}}

\def\b{\bm{\beta}}

\newcommand{\norm}[1]{\left\Vert#1\right\Vert}
\newcommand{\abs}[1]{\left\vert#1\right\vert}

\newcommand{\prob}{\mathbb{P}}
\newcommand{\p}{\mathbf{p}}
\newcommand{\N}{\mathbb{N}}
\newcommand{\R}{\mathbb{R}}

\newcommand{\mH}{\mathcal{H}}

\newcommand{\mN}{\mathcal{N}}

\newcommand{\mO}{\mathcal{O}}

\newtheorem{definition}{Definition}[section]
\newtheorem{lemma}{Lemma}[section]

\newtheorem{theorem}{Theorem}[section]

\newtheorem{remark}{Remark}[section]

 \theoremstyle{assumption}

\usepackage{multicol}

\begin{document}

\twocolumn[

\aistatstitle{Uncertainty Quantification for Sparse Deep Learning}

\aistatsauthor{  Yuexi Wang and Veronika Ro\v{c}kov\'{a} }

\aistatsaddress{  Booth School of Business, University of Chicago } ]

 \begin{abstract}
Deep learning methods continue to have a decided impact on machine
learning, both in theory and in practice.
Statistical theoretical developments have been mostly concerned with
approximability or rates of estimation when recovering infinite
dimensional objects (curves or densities).
Despite the impressive array of available theoretical results, the
literature has been largely silent about {\em uncertainty quantification}
for deep learning.
This paper takes a step forward in this important direction by  taking a
Bayesian point of view.
We study Gaussian approximability of certain aspects of posterior
distributions of sparse deep ReLU architectures in non-parametric
regression.
Building on tools from Bayesian non-parametrics, we provide
semi-parametric Bernstein-von Mises theorems for linear and quadratic
functionals, which guarantee  that   implied
Bayesian credible regions have valid frequentist coverage.   Our results
provide new theoretical justifications for (Bayesian) deep learning with
ReLU activation functions, highlighting their {\em inferential potential}.

 \end{abstract}

\section{Introduction}
Neural networks have emerged as one of the most powerful prediction systems. Their empirical success has been amply documented
 in many applications including image classification \citep{krizhevsky2012imagenet}, speech recognition \citep{hinton2012deep} or game intelligence \citep{silver2016mastering}.
Beyond  algorithmic developments, there has been a rapid progress in theoretical understanding of deep learning \citep{anthony2009neural}. The majority of existing {\em statistical} theory has been concerned with  {\em prediction} aspects, e.g. approximability \citep{telgarsky2016benefits, yarotsky2017error, vitushkin1964proof} or
rates of convergence (either from a frequentist point of view \citep{mhaskar2017and, poggio2017and, schmidt2017nonparametric} or a Bayesian point of view  \citep{polson2018posterior}). 
A distinguishing feature of statistics, that goes beyond mere construction of prediction maps, is providing uncertainty quantification (UQ) for inference (hypothesis testing and confidence assessments). 
 The statistical approach to uncertainty quantification uses observations to construct a random subset (confidence set) which contains the truth with large probability. 
While computational   methods such as Boostrapped DQN \citep{osband2016deep} and Deep Ensembles \citep{lakshminarayanan2017simple} have been proposed to quantify predictive uncertainty, 
theoretically justifiable developments on UQ for deep learning are more rare.


A structured approach to the problem of uncertainty assessment lies in Bayesian hierarchical modeling. 
The Bayesian paradigm for deep learning places a probabilistic blanket over architectures/parameters and allows for  uncertainty quantification via posterior distributions \citep{neal1993bayesian}. 
While exact Bayesian inference is computationally intractable, many approximate methods have been developed including MCMC  \citep{neal2012bayesian}, Variational Bayes \citep{ullrich2017soft}, Bayes by Backprop \citep{blundell2015weight}, Scalable Data Augmentation \citep{wang2019scalable}, Monte Carlo Dropout \citep{gal2016dropout}, Hamiltonian methods \citep{springenberg2016bayesian}. 
The  Bayesian inference is fundamentally justified by the Bernstein-von Mises  (BvM) theorem.
The BvM phenomenon occurs when, as the number of observations increases, the posterior distribution is approximately Gaussian, centered at an efficient estimator of the parameter of interest. Moreover, the posterior credible sets, i.e. regions with prescribed posterior probability, are then also confidence regions with the same asymptotic coverage.  While the BvM limit is not unexpected in regular parametric models, infinite-dimensional notions of BvM are far from obvious 
(see e.g. \citet{castillo2013nonparametric}).

Our paper deals with uncertainty quantification. Our approach is inherently Bayesian and, as such, is conceptually epistemic where uncertainty about the unknown state of nature is expressed through priors and coherently updated with the data. The frequentist notion of uncertainty is primarily aleatoric as it reflects variability in possible realizations of an event that is largely stochastic in nature and is irreducible. The premise of the BvM phenomenon is that these two uncertainties, while qualitatively very different, are not mutually exclusive in the sense that their quantifications can agree. Priors that are not subjective and more automatic do not necessarily adhere to epistemic interpretation and can yield aleatoric measures of quantification. Our work sheds light on the fact that frequentist calibration is an attainable goal of Bayesian statistical procedures, where the BvM phenomenon facilitates communication of uncertainty using the more universally understood frequentist concept \citep{dawid1982well}.  

In this note, we study the {\em semi-parametric} BvM phenomenon concerning the limiting  behavior of the posterior distribution of certain low-dimensional summaries of a regression function.
In particular, we assume a non-parametric regression model with fixed covariates and sparse deep ReLU network priors, which have been recently shown to attain the optimal speed of posterior contraction \citep{polson2018posterior}. Building on
 \citet{castillo2015bernstein}, who laid down the general framework for semi-parametric BvMs, and  on \citet{polson2018posterior}, we formulate asymptotic normality for linear and quadratic functionals. Related semi-parametric BvM  results have been established for density estimation \citep{rivoirard2012bernstein},  Gaussian process priors \citep{castillo2012semiparametric, castillo2012semiparametric2}, covariance matrix \citep{gao2016bernstein} and tree/forest priors \citep{rockova2019bvm}.  
Our results  provide new frequentist theoretical justifications for Bayesian deep learning inference with certain aspects of a regression function.

{

Our analysis focuses on sparse deep ReLU networks. Deep networks have been shown to outperform shallow ones in  terms of representation power \citep{telgarsky2016benefits}, model complexity \citep{mhaskar2017and} and  generalization \citep{kawaguchi2017generalization}.  The ReLU squashing function has been generally preferred due to its expressibility and inherent sparsity. For instance, \citet{yarotsky2017error} provides error bounds for approximating polynomials and smooth functions with deep ReLU networks.  \citet{schmidt2017nonparametric} showed that deep sparse ReLU  networks can yield rate-optimal reconstructions of smooth functions and their compositions. Sparse architectures (in addition to ReLU) can  reduce the test error. For example, sparsification can be achieved with dropout \citep{srivastava2014dropout} which averages over sparse structures by randomly removing nodes and, thereby, alleviates overfitting.  More recently,
\citet{polson2018posterior} proposed  Spike-and-Slab Deep Learning (SS-DL) as a fully Bayesian variant of dropout. Their framework  provably does not overfit and achieves an {\em adaptive} near-minimax-rate optimal posterior concentration. 
\citet{liu2019variable} studies the BvM phenomena for the gradient function of Bayesian deep ReLU network and proposes a variable selection method based on the credible intervals. We continue the theoretical investigation of SS-DL in this paper.
}

Similar to \citet{rockova2019bvm}, we consider a non-parametric regression model where responses $\Y^{(n)}=(Y_1, \ldots, Y_n)'$ are linked to fixed covariates  $\x_i=(x_{i1},\ldots, x_{ip})'\in [0,1]^p$ for $i=1,\ldots, n$  as follows
\begin{equation}\label{eq:dgp}
Y_i=f_0(\x_i)+\epsilon_i, \, \epsilon_i \overset{iid}{\sim} N(0,1),
\end{equation}
where  $f_0 \in \mH^\alpha_p$ is an $\alpha$-H\"{o}lder smooth function on a unit cube  $[0,1]^p$ for some $\alpha>0$. The true generative model implied by (\ref{eq:dgp}) will be denoted by  $\prob^n_0$. We want to reconstruct $f_0$ with $f\in\F$, where the model class $\F$ is assigned a prior distribution  $\Pi$. Our goal is to study the asymptotic behavior of the posterior distribution 
\[
\Pi\left[\sqrt n (\Psi(f)- \hat \Psi)\,|\,\Y^{(n)}\right],
\]
where $\Psi : \F \to \R$  is a measurable function of interest and where $\hat \Psi$ is a random centering point (see Theorem 2.1 in \citet{castillo2015bernstein}). 

Two functionals are considered in our work. The first one is the linear functional 
\begin{equation}\label{eq:Psi}
\Psi(f)=\frac{1}{n}\sum_{i=1}^n a(\x_i) f(\x_i), 
\end{equation}
with a   constant weighting functions $a(\cdot)$. We discuss potential generalizations to the non-constant case later in Section \ref{sec:discussion} . The second functional of interest is the squared-$L^2$ norm 
\begin{equation}\label{eq:Psi_L2}
\Psi(f)=\norm{f}_L^2,\quad 
\end{equation}
where  $\norm{f}_L^2=\frac{1}{n}\sum_{i=1}^n [f(\x_i)]^2$. Note that $\norm{\cdot}_L$ corresponds to the LAN (locally asymptotically normal) norm, which is equivalent to the empirical $L^2$-norm $\norm{\cdot}_n$ in our model.  There is extensive literature on minimax estimation of linear and quadratic functionals, initiated in \citet{ibragimov1985nonparametric} and followed by \citet{cai2005adaptive, efromovich1996optimal, collier2017minimax}, to name a few. While the linear functional is useful for inference about the average  regression surface, the quadratic functional is useful in many testing problems, including construction of confidence balls \citep{cai2006adaptive} and goodness of fit tests \citep{dumbgen1998new, butucea2007goodness}. We study  adaptive estimation of the two functionals from a Bayesian perspective.

First, we give the definition of asymptotic normality. 
\begin{definition}
Denote with  $\beta$ the bounded Lipschitz metric for weak convergence and  with $\tau_n$ the mapping $\tau_n: f \to \sqrt n (\Psi(f)-\Psi_n)$.
We say that the posterior distribution of the functional $\Psi(f)$ is asymptotically normal with centering $\Psi_n$ and variance $V$ if 
\begin{equation}\label{eq:asy_normal}
\beta(\Pi[\cdot\mid \Y^{(n)}]\circ \tau_n^{-1}, \mN(0, V))\to 0,
\end{equation}
in $\P_0^n$-probability  as $n\to \infty$. We  will write  this more compactly as $\Pi[\cdot\mid \Y^{(n)}]\circ \tau_n^{-1} \rightsquigarrow \mN(0, V)$.
\end{definition}
Next, we say that the posterior distribution {\em satisfies the BvM theorem} if (\ref{eq:asy_normal}) holds with $\Psi_n=\hat\Psi+o_P(\frac{1}{\sqrt n})$ for $\hat \Psi$ a linear efficient estimator of $\Psi(f_0)$. 

\citet{castillo2015bernstein} provide general conditions on the model and on the function $\Psi(\cdot)$ to guarantee that the BvM phenomenon holds. Our results are built on the {first-order} approximation technique  developed in  their work. Essentially, we want to show that  the sparse deep learning posterior can approximate both $f_0$ and the linear expansion term well enough so that the remainder term vanishes when $n \to \infty$. 

The rest of our paper is organized as follows.  Section \ref{sec:deep_relu} defines sparse ReLU networks and reviews the posterior concentration results. Section \ref{sec:main_thm} contains the main results of BvM properties of  the two functionals and Section \ref{sec:adaptive} discusses extensions to adaptive priors. Section \ref{sec:discussion} concludes with a discussion.

\section{Deep ReLU Networks}\label{sec:deep_relu}
We follow the notation used in \citet{polson2018posterior}. We denote with $\F(L,\p,s)$  the class of sparse ReLU networks with $L\in \N$ layers,  a vector of $\p=(p_0, \ldots, p_{L+1})'\in\N^{L+2}$ hidden units and sparsity level $s\in \N$, which is the upper bound on the number of nonzero parameters. In our model, we have $p_0=p$ and $ p_{L+1}=1$. Each function $f_{\B}^{DL}(\x) \in \F(L,\p,s)$ takes the form
\begin{equation}\label{eq:DL}
f_{\B}^{DL}(\x) = W_{L+1}\sigma_{b_L}\left(W_L \sigma_{b_{L-1}} \cdots \sigma_{b_1}(W_1\x)\right)+b_{L+1}
\end{equation}
where $b_l\in \R^{p_l}$ are shift vectors and $W_l$ are  $p_l\times p_{l-1}$ weight matrices that link neurons between the $(l-1)^{th}$ and $l^{th}$ layers and $\sigma_{b}(\x)$ is the squashing function. 
Throughout this work, we assume the {\em rectified linear (ReLU)} function $\sigma_{b}(\x) = \max(\x+b,0)$ which applies to vectors elementwise. Note that the top layer shift parameter $b_{L+1}$ is {\em outside} the ReLU function since the top layer is only a linear function. We denote the sets of all model parameters with
\begin{equation}\label{eq:B}
\B=\{ (W_1, b_1),\ldots,(W_L, b_L), (W_{L+1}, b_{L+1}) \}.
\end{equation}

Let $Z_l \in \R^{p_l}$ represent the hidden nodes of the $l^{th}$ layer obtained as 
\begin{align*}
Z_{l}(\x)&=\sigma_{b_l}(W_l Z_{l-1}(\x)), \quad \text{ for }\quad l=1\ldots, L,\\
Z_0(\x)&=\x.
\end{align*}
We use $Z=\{Z_l\}_{l=1}^L$ to represent the collection of all hidden neurons. Their values are completely determined by $\{W_l, b_l \}_{l=1}^L$, independently of the top layer parameters $\{W_{L+1}, b_{L+1}\}$.

\subsection{Spike-and-Slab Priors}
We place a probabilistic structure on $\B$ that is slightly different from \citet{polson2018posterior}. In particular, we remove the spike-and-slab prior on the top layer $L$ to obtain a fully-connected top layer for each function $f_{\B}^{DL}(x)$. Such  a relaxation on the top layer facilitates  the \textit{change of measure} step in our results.  Later we show that having a  fully connected top layer  {\em does not} affect the network approximability and the posterior concentration rate. 

We convert $\B$ into a vector by stacking $\{W_l, b_l\}_{l=1}^{L+1}$  from the bottom to the top  and denote  $\B=(\beta_1, \ldots, \beta_T)' $, where $T=\sum_{l=0}^Lp_{l+1}(p_l+1)$ is the number of parameters in a fully connected network with $L$ layers and a vector of $\p$ neurons. Note that $\{\beta_j\}_{j>T-(p_L+1)}$ corresponds to the top layer $\{W_{L+1}, b_{L+1}\}$.   Then the priors on $\B$ are
\begin{equation}\label{eq:beta}
\pi(\beta_j\mid \gamma_j)=\gamma_j \tilde \pi(\beta_j)+(1-\gamma_j)\delta_0(\beta_j),
\end{equation}
with
\begin{equation}\label{eq:gamma}
\gamma_j=1 \quad \text{ for }\quad  j> T-(p_L+1),
\end{equation}
where $\tilde\pi(\beta)$ is specified as
\begin{align}\label{eq:beta_prior}
\tilde \pi(\beta_j) =\left\{\begin{array}{cc}
N(0,1), & j>T-{p_L+1},\\
\mathrm{Uniform}[-1,1], & j \leq T-{p_L+1},
\end{array} \right.
\end{align}
 i.e., the top layer weights follow standard normal distribution, while the deep weights follow uniform distribution on $[-1,1]$. 
 $\delta_0(\beta)$ is a dirac spike at zero, and $\gamma_j\in\{0,1\}$ for whether or not $\beta_j$ is nonzero. We let $\gamma_j=1$ for all $j> T-(p_L+1)$ so that the top layer is fully connected.  The vector $\gamma=(\gamma_1,\ldots, \gamma_T)'$ encodes the connectivity pattern below the top layer. We assume that, given the network structure and the sparsity level $s=\abs{\gamma}> p_L$, all architectures are  equally likely a priori, i.e.
\begin{equation}\label{eq:gamma_prior}
\pi(\gamma\mid \p,s)= \frac{\1( \gamma_j=1 \,\,\text{for}\,\, j > T-p_L-1)}{{{T-p_L-1} \choose {s-p_L-1}}}.
\end{equation}
We denote with $\mV^{\p,s}$ the set of all combinatorial possibilities of  connectivity patterns below the top layer. For a given sparsity level $s$,  we can write
\begin{equation}\label{eq:shell}
\F(L,\p,s) = \bigcup_{\gamma\in \mV^{\p,s}}\F(L,\p,\gamma),
\end{equation}
where each shell  $\F(L,\p,\gamma)$ consists of all uniformly bounded functions $f_{\B}^{DL}$ with the same connectivity pattern $\gamma$, i.e. $ \F(L,\p,\gamma)=\{f_{\B}^{DL}(\x) \in \F(L,\p,s): f_{\B}^{DL}(\x) \text{\,\, as in \eqref{eq:DL} with $\B$ arising from (\ref{eq:beta}) for a given}\\ \,\, \gamma \in\mV^{\p,s} \text{and where $\norm{f_{\B}^{DL}(\x)}_\infty<F$}\}$ for some $F>0$.

\begin{remark}
The prior for the deep coefficients $\beta_j$ in (\ref{eq:beta_prior}) can  be replaced by 
\begin{equation}\label{eq:beta_all_normal}
\tilde\pi (\beta_j)=N(0,1), \forall j=1, \ldots, T.
\end{equation}
The posterior concentration rate can be also shown to be rate-optimal under this prior. We give the  sketch  of the proof after Theorem \ref{thm:post_concentrate} in Supplemental Material. Moreover, the BvM property for this prior can be immediately concluded from our proofs of Theorems 3.1-3.3.

\end{remark}

\subsection{A Connection between Deep ReLUs and Trees}\label{sec:sum}
Before proceeding,  it will be useful to revisit a connection between networks and trees.
Recall that any deep ReLU network function can be written as a sum of local linear functions, i.e.
\begin{equation}\label{eq:local_linear}
f_{\B}^{DL}(\x) = \sum_{k=1}^K \1( \x \in \Omega_k) (\tilde \beta_k' \x+\tilde\alpha_k),
\end{equation}
where $\{\Omega_k\}_{k=1}^K$ is a  partition of the predictor space made by recursive ReLU layers (see \citet{polson2017deep} for illustrations).  Both the partition $\{\Omega_k\}_{k=1}^K$ and the coefficients of the local linear functions $\{\tilde\beta_k, \tilde\alpha_k\}_{k=1}^K$ are determined from $\{W_l,b_l\}_{l=1}^{L+1}$. We have omitted the dependence on $\B$ for simplicity of notation.

\citet{balestriero2018mad} view ReLU as Max-Affine Spline Functions (MASO) and describe how the local linear functions and partitions are determined from weights $\B$. They point out that the partition by layer $l$ contains up to $2^{p_l}$ convex conjoint regions. In practice, however, many of them could be empty intersections.  \citet{montufar2014number} shows that the number of linear regions $K$ of ReLU networks  is upper-bounded by $2^T$ and lower-bounded by $(\prod_{l=1}^{L-1}\floor{\frac{p_l}{p}}^{p})\sum_{j=1}^{p}{p_L\choose j}$. \citet{hanin2019complexity} further measure the volume of the boundaries between these regions.

Deep ReLU networks are similar to trees/forests methods in the sense that  they also partition the predictor space. 
In fact, any regression tree can be represented by a  neural network with a particular activation function, as we illustrate below using  an example from \citet{biau2016neural}.

\paragraph{Example 1} Define an activation function $\tau_{b}: \R \to \{-1, 1\}$ such that
\[
\tau_b(x)=2\1_{x+b\geq 0}-1.
\]
We can reconstruct a two-dimensional ($p=2$) example in Figure \ref{fig:tree} with a neural network as
\begin{align*}
Z_1&=\tau_{-b_1}(X_1) &&Z_2=\tau_{-b_2}(X_2),\\
Z_3&=\tau_{-2}(-Z_1+Z_2)&&Z_4=\tau_{-2}(Z_1+Z_2), \\
Z_5& =\tau_{-1}(Z_1)&& f_{\B}^{DL}(\x)= \sum_{i=3}^5 W_i Z_i.
\end{align*}
where $b_1$ and $b_2$ set the decision boundaries along $(X_1, X_2)$ axes in the tree, and $\{W_i\}_{i=3}^5$ are the jump sizes in each leaf node. A more detailed explanation of the choice of weights can be found in \citet{biau2016neural}. By analogy, the hierarchical segmentation is determined by the deep layers while the values of the leaf nodes are assigned by the top layer. 

\begin{figure}[!ht]
\begin{center}
\includegraphics[width=0.38\textwidth]{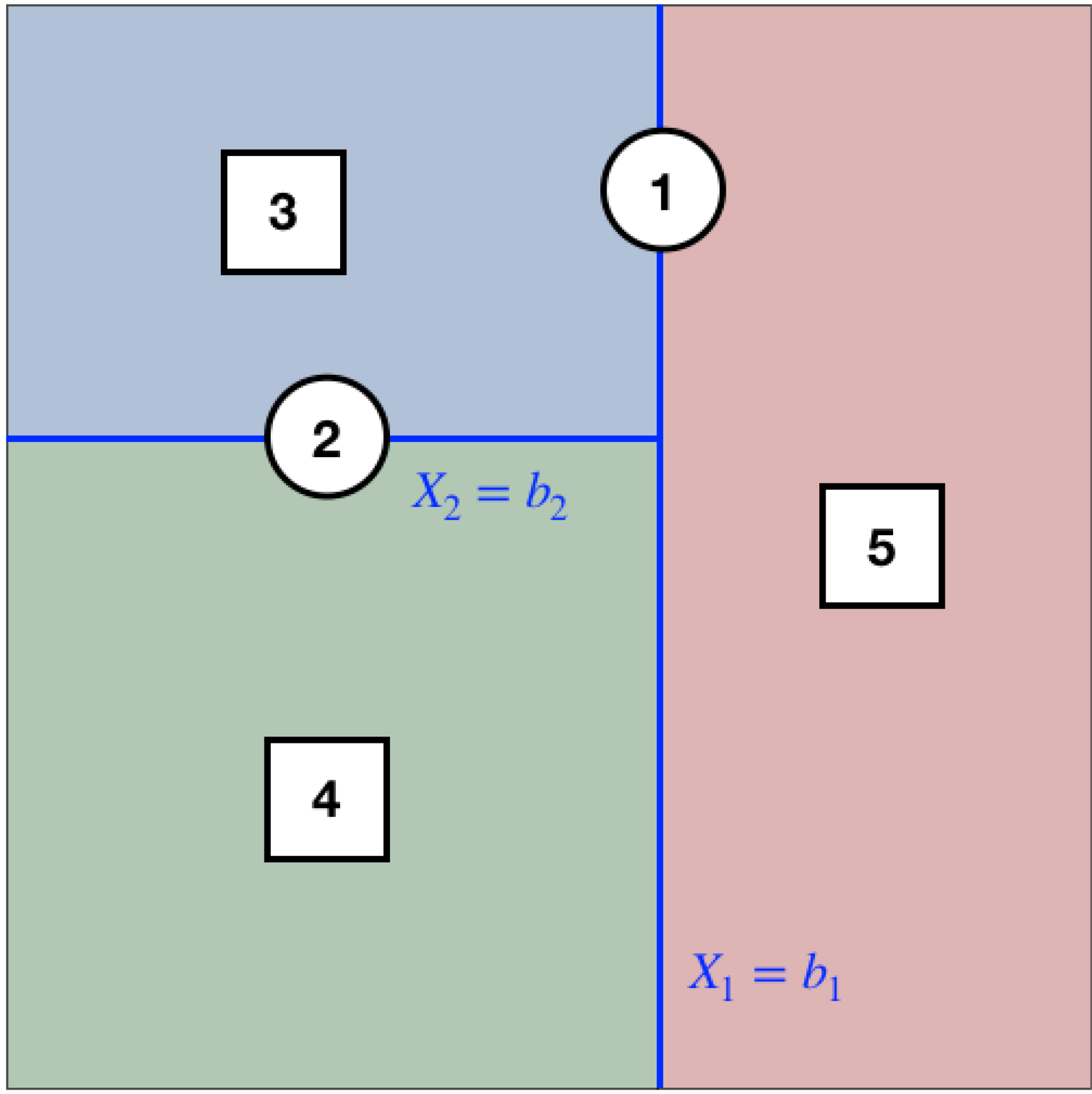}
\caption{Visualization of Example 1}\label{fig:tree}
\end{center}
\end{figure}

Deep ReLU networks use a different activation function and thereby place fewer restrictions on the geometry of the partition boundaries (shards as opposed to boxes). 
There are two aspects that make the analysis of deep ReLU networks more difficult. First, the partitioning lines do not align with coordinate axes when $W_l\neq 0$. Second, the partitioning cells $\{\Omega_k\}_{k=1}^K$ and the local linear coefficients $\{\tilde\beta_k, \tilde\alpha_k\}_{k=1}^K$ are related as they both depend on the unknown coefficients $\{W_l, b_l\}_{l=1}^L$. In tree models, on the other hand, they are independent parameters. 

To illustrate the correspondence between the partitions and local linear functions as well as  their relationship to $\B$, we consider the following toy example.

\paragraph{Example 2} Consider $L=1, p=2$ and $p_1=2$.  Given the weights and shifts as
\[
W_1=\left(\begin{array}{c}
W_1^1\\
W_2^1
\end{array}
\right), b_1=(b_1^1, b_2^1), W_2=\left(\begin{array}{c}
W_1^2\\
W_2^2
\end{array}
\right), b_2=b^2,
\]
we can write the model as
\begin{align*}
Z_1&=\sigma_{b^1_1}(W^1_1 \x), \quad Z_2=\sigma_{b^1_{2}}(W^1_{2} \x),  \\
 f_{\B}^{DL}(\x)&=\sigma_{b^2}(W^2_1Z_1+W^2_2Z_2).
\end{align*}
Then the corresponding $\{\tilde \beta_k,\tilde \alpha_k, \Omega_k\}_{k=1}^5$ for each local linear function can be organized as

\scalebox{0.8}{
\centering
\begin{tabular}{cccc}
\toprule
i & $\tilde \beta_i$ &$\tilde \alpha_i$ & $\Omega_i$ \\
\midrule
1 & $W^2_1W^1_1+W^2_2W^1_2$ & $W_1^2b^1_1+W_2^2b^1_2+b^2$ & $A_1\cap A_2 \cap A_3$\\
2 & $W^2_1W^1_1$ & $W^2_1b^1_1+b^2$ & $A_1\cap A_2^c \cap A_4$\\
3 & $W^2_2W^1_2$ & $W_2^2b^1_2+b^2 $ & $A_1^c\cap A_2 \cap A_5$ \\
4 & $0$ & $\max(b^2,0)$ &$A_1^c\cap A_2^c$ \\
5 & $0$ & $0$ & $( \Omega_1\cup\Omega_2\cup\Omega_3\cup\Omega_4)^c$\\
\bottomrule
\end{tabular}
}

with
\begin{align*}
A_1&=\{\x: W^1_1 \x+b^1_{1} >0\}, &&A_2 =\{\x: W^1_{2} \x+b^1_{2} >0\},\\
A_3&=\{\x:\tilde \beta_1 \x+\tilde\alpha_1>0\}, &&A_4=\{\x: \tilde\beta_2 \x+\tilde\alpha_2 >0\}, \\
 A_5& =\{\x:\tilde\beta_3 \x+\tilde\alpha_3>0\}.
\end{align*}
Here we use $A_i^c$ to denote the complement of set $A_i$, i.e., $A_i^c=\{\x \in\R^2: \x \notin A_i\}$. The covariance matrix of $\{\tilde\beta_k\}_{k=1}^3$ is
\[
\Var \left(\begin{array}{c}
\tilde\beta_1\\
\tilde\beta_2\\
\tilde\beta_3
\end{array}
\right) = \frac{2}{9}\left(\begin{array}{ccc}
2&1 & 1 \\
1& 1&0\\
1& 0 &1
\end{array}
\right).
\]
This example is plotted in Figure \ref{fig:toy}, where the boundaries of the partitions are nested according to $\{\tilde \beta_k, \tilde \alpha_k\}_{k=1}^5$ and  determined by $\{W_l,b_l\}_{l=1}^2$.
\begin{figure}[!ht]
\begin{center}
\includegraphics[width=0.5\textwidth]{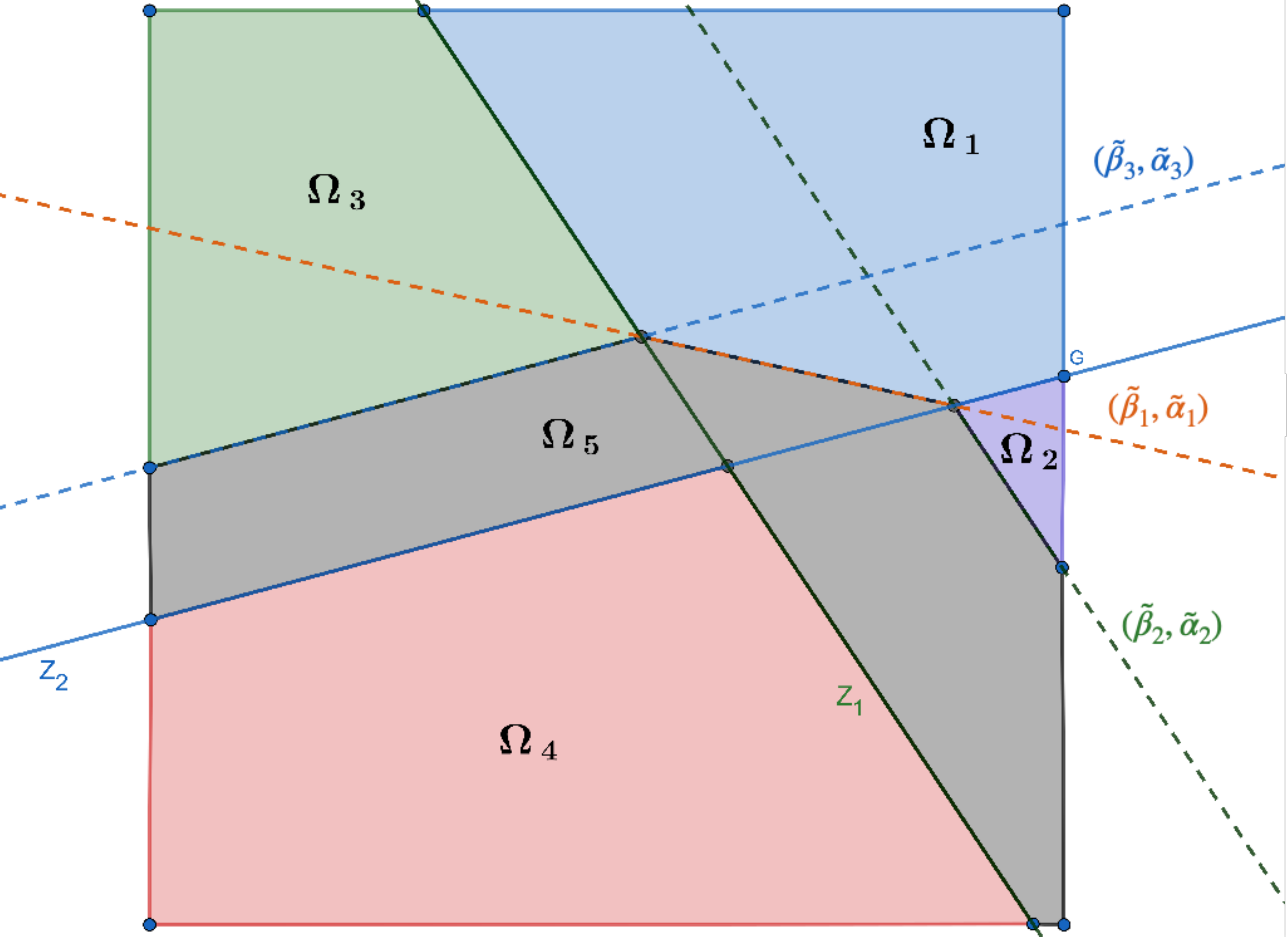}
\caption{Visualization of Example 2}\label{fig:toy}
\end{center}
\smallskip
\footnotesize
\end{figure}

\subsection{Posterior Concentration}\label{sec:new_prior}
One essential prerequisite for our BvM analysis is  optimal rate of posterior convergence.
\citet{polson2018posterior} (PR18) showed that  sparse deep ReLUs  attain the  near-minimax optimal rate and are {\em adaptive} to unknown smoothness under suitable priors on the architecture size. 
%
Here, we use a modified prior with a fully connected top linear layer (as given by  (\ref{eq:gamma})). The posterior concentration result  still holds.
Indeed, for an arbitrary sparse network, there exists at least one network  with a fully connected linear layer that achieves the same approximation error. The approximability of our class of networks is thus the same as the class considered in PR18. We illustrate how such a network can be constructed in the Supplemental Material (Lemma \ref{lem:approx_error}).

Denoting  $(L^*, N^*, s^*)$ as  in Theorem 5.1 of PR18 and choosing the parameters of the network as 
\begin{equation}\label{eq:An_cond}
 \left\{\begin{array}{l}
 L=L^* +1 \asymp \log(n),\\
 s=s^*+24 pN^* \lesssim n^{p/(2\alpha+p)}, \\
 \end{array}\right.
 \end{equation}
 we define
\begin{equation}\label{eq:An}
A_n^M=\{f_{\B}^{DL}\in  \F(L,\p,s): \norm{f_{\B}^{DL}-f_0}_L \leq M\xi_n \}
\end{equation}
with $\xi_n=n^{-\alpha/(2\alpha+p)}\log^\delta(n)$ for some $M>0$ and $\delta>0$.
As we formalize in Theorem \ref{thm:post_concentrate} in the Supplement, 
one can show $\Pi[A_n^{M_n}|\Y^{(n)}]=1+o_P(1)$ for any $M_n\to \infty$ and uniformly bounded $\alpha$-H\"{o}lder mappings $f_0$.

Our analyses in Section \ref{sec:main_thm} will be performed locally  on sets  $A_n^{M_n}$ where the posterior concentrates.

\section{Semi-parametric BvM's}\label{sec:main_thm}
Locally on the sets $A_n\equiv A_n^{M_n}$ we will perform expansions of the log-likelihood as well as the functional $\Psi$.
The log-likelihood is denoted with 
\[
\ell_n(f)=-\frac{n}{2}\log 2\pi-\sum_{i=1}^n \frac{[Y_i-f(\x_i)]^2}{2}.
\]
and the log-likelihood ratio $\Delta_\ell(f)=\ell(f)-\ell(f_0)$ can be expressed as a sum of a quadratic term and a stochastic term via the LAN expansion as follows
\[
\Delta_\ell(f) =-\frac{n}{2}\norm{f-f_0}^2_L+\sqrt n W_n(f-f_0)
\]
where
\begin{align*}
W_n (f-f_0)&=\expect{f-f_0, \sqrt n \boldsymbol{\epsilon}}_L\\
&=\frac{1}{n} \sum_{i=1}^n \sqrt n \epsilon_i [f_0(\x_i)-f(\x_i)].
\end{align*}
We focus on the first-order approximations of the functionals.
For any $f \in A_n$, we write
\[
\Psi(f)=\Psi(f_0)+\expect{\Psi_0^{(1)}, f-f_0}_L+r(f,f_0).
\]
 The first-order term $\Psi_0^{(1)}$ is equal to $a$ for linear functionals \eqref{eq:Psi} and $2f_0$ for the quadratic functional  \eqref{eq:Psi_L2}. The inner product $ \expect{\cdot,\cdot}_L$ is defined as
$
\expect{g,h}_L = \frac{1}{n} \sum_{i=1}^n g(\x_i)h(\x_i)
$
 for  two functions $g$ and $h$. 
 
Before we dive into the main development, we recall the  results in \citet{castillo2015bernstein} which will be leveraged in our analysis. 

There are  two sufficient conditions for obtaining weak asymptotic normality as defined in (\ref{eq:asy_normal}). The first one is the vanishing remainder 
\begin{equation}\label{eq:rn}
\sup_{f\in A_n} \abs{t\sqrt n r(f, f_0)}=o_P(1).
\end{equation}
The second one is verifying
\begin{equation}\label{eq:c_o_m}
\frac{\int_{A_n} e^{\ell_n(f_t)-\ell_n(f_0)}d\Pi(f)}{\int_{A_n} e^{\ell_n(f)-\ell_n(f_0)}d\Pi(f)}=1+o_P(1),  \forall t\in \R, 
\end{equation}
 where $ f_t=f-\frac{t\Psi_0^{(1)} }{\sqrt n}$.
 
The second condition in (\ref{eq:c_o_m}) can be shown with a \textit{change of measure} argument and it guarantees that the posterior has no extra bias term.
With these two conditions satisfied, the posterior behavior of  $\sqrt n (\Psi(f)-\hat\Psi)$ is asymptotically mean-zero normal with variance $V_0=\norm{\Psi_0^{(1)}}_L^2$, where
\[\hat\Psi = \Psi(f_0) +\frac{W_n(\Psi_0^{(1)})}{\sqrt n}\]
is a random centering point.

A crucial step   is  performing the \textit{change of measure}  in  (\ref{eq:c_o_m}), where  we replace $f$ with a shifted function $f_t$ in the integration.
This is complicated by the fact that the  shifted function $f_t$ does not necessarily have to correspond to a deep ReLU network from the class $\F(L,\p, s)$. In the analysis of trees, for instance, one can condition on the partition parameter and perform the shift of measure on functions supported on the {\em same} partition, where the shift only affects step heights. For a deep ReLU network, however, partitions $\Omega_k$ and local linear coefficients $(\wt b_k,\wt \alpha_k)$ in \eqref{eq:local_linear} are not independent as they  {\em both} depend on the deep weights $\{W_l, b_l\}_{l=1}^L$. It is thereby   not obvious how the shift affects the partitions and the network coefficients.
If we want to preserve the partitions of the predictor space, the only ``free" parameters left to play with are the top layer weights $\{W_{L+1}, b_{L+1}\}$. Similarly as for trees, we consider conditioning on the {\em deep} coefficients $\{W_l, b_l\}_{l=1}^L$, which is equivalent to conditioning on $\gamma$ and $ Z=\{Z_l\}_{l=1}^L$, and perform the change of measure only on the top layer. We write the function class conditionally on $(\gamma, Z)$ as 
\begin{multline}
\F(L,\p, \gamma, Z)= \{ f\in \F(L,\p,s):  f=W_{L+1}Z_L+b_{L+1} \\
\text{and $f$ has connectivity $\gamma$}\}.
\end{multline}
Since the prior of $\{W_l, b_l\}_{l=1}^L$ is continuous, there are  infinitely many $(\gamma, Z)$-dependent shells $\F(L,\p,\gamma, Z)$ inside $\F(L,\p, s)$. The general scheme of our proof is as follows. First,  for each shell $\F(L,\p, \gamma, Z)$, we have a local centering point $\hat \Psi_Z^\gamma$ and a local variance $V_Z^\gamma$. Moreover, the shifted function $f_t$ inside each shell lives on the {\em same partition} as $f$ and the change of measure can therefore be performed more easily. Second, we show that  $\hat \Psi_Z^\gamma$ and  $V_Z^\gamma$ converge {\em uniformly} to a global centering point $\hat \Psi$ and a global variance  $V_0$ for all $Z$ and $\gamma$ inside $A_n$. This implies that  we recover the global BvM on $\F(L,\p, s)$. The details of the local projections and the proof of all theorems are  in Supplemental Material.

\subsection{Linear Functionals}
To start,  we consider the linear functional  in  (\ref{eq:Psi}) where $a(\cdot)$ is a constant function in which case  $\Psi(f)$ can be viewed as a constant multiple of the average regression surface evaluated at $\{\x_i\}_{i=1}^n$. Let
\begin{align*}
&\Psi(f)=\Psi(f_0)+ \expect{a, f-f_0}_L, \quad \Psi_0^{(1)}=a, \\
&\hat \Psi=\Psi (f_0)+\frac{W_n(a)}{\sqrt n}, \quad V_0=\norm{a}_L^2.
\end{align*}

\begin{theorem}\label{thm:linear}
Assume the model (\ref{eq:dgp}),  where $f$  is endowed with a prior on $F(L, \p, s)$ defined in (\ref{eq:beta}), (\ref{eq:gamma}) and \eqref{eq:beta_prior}. Assume that  (\ref{eq:An_cond}) is satisfied and that $f_0 \in \mH_p^\alpha$, where $p=\mO(1)$ as $n\to \infty$, $\alpha<p$ and $\norm{f_0}_\infty\leq F$. When $a(\cdot)$ is  constant,  we have
\[\Pi(\sqrt n (\Psi(f)-\hat \Psi)\mid \Y^{(n)}) \rightsquigarrow N(0, \norm{a}_L^2) \]
in $\prob_0^n$-probability as $n\to\infty$.
\end{theorem}
\proof Reference to a  Section \ref{thm:proof_linear} in Supplemental Material.
When $a(\cdot)$ is constant, the shifted functions  $f_t$ can be easily constructed by shifting the top intercept $b_{L+1} \to b_{L+1} -\frac{t a}{\sqrt n}$. The projection of $a$ is not needed as the remainder term is zero.

\begin{remark}\label{remark:non-constant} When $a(\cdot)$ is not constant, we need the projection of $a(\cdot)$ (conditional on $(\gamma, Z)$), denoted by $a^\gamma_{[Z]}$, to be close to $a$ for all $Z$ and $\gamma$ supported by $A_n$. In order for the BvM result to hold,  we would then require the {\em no-bias} condition 
\begin{equation}\label{eq:a_no_bias}
\expect{a-a^\gamma_{[Z]}, f-f_0}_L=o_P\left(\frac{1}{\sqrt n}\right).
\end{equation}
In order to verify this condition, one could view $Z$ as  a collection of random sparse ReLU features and study the approximability of this class. Although there are some studies  on the universal approximation error of random ReLU features \citep{sun2018random, yehudai2019power}, general conditions for the approximation ability of such projections are not yet obvious.
\end{remark}

\subsection{Squared $L^2$-norm Functional}
We consider the quadratic functional   in  (\ref{eq:Psi_L2}).  The estimation of the $L^2$-norm is closely related to minimax optimal testing of hypothesis under empirical $L^2$ distance \citep{collier2017minimax}. This functional could serve as the risk function and has been used in many testing problems \citep{cai2006adaptive,dumbgen1998new}. The next theorem relies on the following notation
\begin{align*}
&\Psi(f)=\Psi(f_0)+ 2\expect{f_0, f-f_0}_L+\norm{f-f_0}^2_L, \Psi_0^{(1)}=2 f_0, \\
&\hat \Psi=\Psi (f_0)+\frac{2W_n(f_0)}{\sqrt n},   V_0=4\norm{f_{0}}_L^2.  
\end{align*}

\begin{theorem}\label{thm:L2}
Assume the model (\ref{eq:dgp}),  where $f$  is endowed with a prior on $F(L, \p, s)$ defined in (\ref{eq:beta}), (\ref{eq:gamma}) and \eqref{eq:beta_prior}. Assume that (\ref{eq:An_cond}) is satisfied and that $f_0 \in \mH_p^\alpha$, where $p=\mO(1)$ as $n\to \infty$, $\alpha \in (\frac{p}{2}, p)$ and $\norm{f_0}_\infty\leq F$. Then we have
\[\Pi(\sqrt n (\Psi(f)-\hat \Psi)\mid \Y^{(n)}) \rightsquigarrow N(0, 4\norm{f_{0}}_L^2) \]
in $\prob_0^n$-probability as $n\to\infty$.
\end{theorem}
\proof Reference to Section \ref{thm:proof_L2} in Supplemental Material.
For this quadratic functional, we use the $(\gamma,Z)$-dependent projection $f_{0[Z]}^{\gamma}$ to approximate $\Psi_0^{(1)}=2f_0$ so that the  {\em change of measure} can be conducted through $\{W_{L+1}, b_{L+1}\}$. The additional constraint $\alpha>p/2$ is added to obtain $\xi_n^2=o(\frac{1}{\sqrt n})$, which ensures that the remainder term \eqref{eq:rn} vanishes.

\section{Adaptive Priors}\label{sec:adaptive}
The results in previous section are predicated on the assumption that the smoothness $\alpha$ is {\em known}. This is hardly ever satisfied in practice and the next natural step is to inquire whether similar conclusions can be obtained when $\alpha$ is unknown. Similarly as PR18, instead of the $\alpha$-dependent choices of the width $N$ and sparsity level $s$ in \eqref{eq:An_cond}, we deploy the following priors that adapt to smoothness
\begin{align}
\pi(N) &=\frac{\lambda^N}{(e^\lambda-1)N!},  \text{ for } \lambda\in \R, \label{eq:N_prior}\\
\pi(s) &\propto e^{-\lambda_s s},  \text{ for } \lambda_s>0.\label{eq:s_prior}
\end{align}
The parameter space now consists of shells of sparse ReLU networks with different widths and sparsity levels, i.e.
\begin{equation}
\F(L)=\bigcup_{N=1}^\infty \bigcup_{s=0}^T \F(L,\p_N^L, s),
\end{equation}
where $\F(L,\p_N^L, s)$ was defined in \eqref{eq:shell}.
An approximating sieve can be constructed that consists of sparse and not so wide networks, i.e.
\begin{equation}
\F_n=\bigcup_{N=1}^{N_n} \bigcup_{s=0}^{s_n} \F(L,\p_N^L, s)
\end{equation}
with $N_n \asymp n\xi_n^2/\log n$ and $s_n\asymp n\xi_n^2$.

Following the same strategy as in the proof Theorem 6.2 of PR18, we extend the posterior concentration  result to the case of adaptive priors  (\ref{eq:beta}), (\ref{eq:gamma}), \eqref{eq:N_prior} and \eqref{eq:s_prior}
(see Theorem \ref{thm:adaptive} in the Supplemental Material). 
The next step is extending the BvM results from the previous section. The following Theorem shows that one can obtain asymptotic normality of the quadratic and  linear functionals without the exact knowledge of  $\alpha$.

\begin{theorem}\label{thm:adaptive_bvm}
Assume the model (\ref{eq:dgp}),  where $f$  is endowed with a prior on $F(L)$ defined through  (\ref{eq:beta}), (\ref{eq:gamma}), \eqref{eq:beta_prior}, \eqref{eq:N_prior} and \eqref{eq:s_prior} with    $L\asymp \log(n)$. 
Assume that $f_0 \in \mH_p^\alpha$, where $p=\mO(1)$ as $n\to \infty, \alpha<p$ and $\norm{f_0}_\infty\leq F$. 
\begin{enumerate}[label=(\roman*)]
\item For the linear functional $\Psi(f)$ in (\ref{eq:Psi}) where $a(\cdot)$ is  constant, we obtain
\[
\Pi(\sqrt n (\Psi(f)-\hat \Psi)\mid \Y^{(n)}) \rightsquigarrow N(0, \norm{a}_L^2),
\]
where $\hat \Psi=\Psi(f_0)+\frac{1}{\sqrt n}W_n(a)$.
\item For the square $L^2$-norm functional $\Psi(f)$ in (\ref{eq:Psi_L2}),  we obtain for $\alpha \in (\frac{p}{2}, p)$
\[
\Pi(\sqrt n(\Psi(f)-\hat \Psi)\mid \Y^{(n)}) \rightsquigarrow N(0, 4\norm{f_0}_L^2)
\]
where $\hat \Psi=\Psi(f_0)+\frac{2}{\sqrt n} W_n(f_0)$.
\end{enumerate}
\end{theorem}
\proof Reference to Section \ref{thm:proof_adaptive} in Supplemental Material.
\begin{remark}
Similar constraints on the smoothness $\alpha$ have been imposed in other related works \citep{farrell2018deep}. However, unlike in other
 developments \citep{schmidt2017nonparametric, farrell2018deep}, the convergence rates we build on are {\em adaptive} in the sense
 that, beyond the assumption $\alpha < p$, the exact knowledge of $\alpha$ is {\em not required}. 
 When the imposed smoothness assumptions do not hold, one could still obtain asymptotic normality via misspecified BvM-type results \citep{kleijn2012bernstein} but uncertainty quantification with the
 implied credible sets would be problematic.
\end{remark}

\begin{remark}
It is worth noting that our results do not hinge on the assumption that $f_0$ came from the prior. Instead, $f_0$ is an arbitrary H\"older
smooth function, not necessarily a neural network. While the model is ultimately mis-specified, our results are attainable due to the expressibility of deep
 ReLU networks where one can approximate $f_0$ with deep  learning mappings with a rapidly vanishing error. The
 fact that our posterior concentrates around the truth at the optimal rate makes the derivation of BvM and valid inference feasible.
\end{remark}

\section{Discussion}\label{sec:discussion}

In this paper, we obtained asymptotic normality results for linear  and squared $L^2$-norm functionals for deep, sparse ReLU networks.  These results can be used as a basis for semi-parametric inference and can be extended in 
various ways.

First, one could obtain similar formulations for general smooth linear functionals by verifying  the {\em no bias}  condition in  (\ref{eq:a_no_bias}). This relates to the approximation ability of random ReLU features  mentioned in Remark \ref{remark:non-constant}. The ReLU features act similarly as random rotational trees. However,  the nested nature of partitions and local linear functions make the analysis difficult. Random features have gained much attention recently. For instance, \citet{rahimi2008random} show how random features can be connected to kernel methods.  
\citet{sun2018random} discuss the universal approximation bounds for compositional ReLU features. \citet{huang2006universal} and  \citet{huang2014insight} provide similar results and they propose an implementation of the extreme learning machine implementation,  where only the top layer is trained while deep layers are sampled randomly from some distribution. A time-series variant of this algorithm is the Deep Echo State Network \citep{sun2017deep, mcdermott2018hierarchical}. 

Another way to obtain BvM for smooth linear functionals would be to construct a less-restrictive projection of the first-order term $\Psi_0^{(1)}$. \citet{schmidt2017nonparametric} shows that parallelization can be realized using embedding networks.  The shifted function $f_t$ could  be constructed as an embedding network that simultaneously represents $(f, \Psi_0^{(1)})$. This representation could  leverage the approximability of smooth functions $a$ with deep neural networks.


To sum up, our semi-parametric BvM results certify that (semi-parametric) inference with Bayesian deep learning
 is valid and that meaningful uncertainty quantification is attainable. Possible applications of our results include
 casual inference, whereby embedding our model within a missing data framework \citep{ray2018semiparametric}, the
 average functional can be used for average treatment effect estimation. In this vein, our results are relevant for the
 development/understanding of the widely sought after machine learning methods for causal inference \citep{athey2017efficient}. In particular, an extension of our work along these lines will constitute a fully-Bayesian variant of
 the doubly-robust plug-in approach of \citet{farrell2018deep}. In addition, the main theorems (Theorem 3.1-3) provide foundations for testing hypotheses such as exceedance of a level $\sum_{i=1}^n f_0(x_i) >c$.
 Lastly, an important future direction will be quantifying uncertainty about  the {\em entire function} $f_0$ (not only its functionals), which was recently formalized for Bayesian CART   by \citet{castillo_rockova}.

 Our work is primarily concerned with theoretical frequentist study of the
 posterior distribution. Investigating practical usefulness and computation of our priors is an important future direction. 
 There are various ways to approximate aspects of  deep learning posterior distributions under spike-and-slab prior, see \citet{polson2018posterior} for a discussion on possible implementations. In addition, \citet{deng2019adaptive} proposed an adaptive empirical Bayesian method for sparse deep learning with a self-adaptive spike-and-slab prior.


\section*{Acknowledgements}
The authors gratefully acknowledge the support from the James S. Kemper Faculty Fund at the Booth School of Business and the National Science Foundation (Grant No. NSF DMS-1944740)

\bibliography{ss-elm}

\newpage
\pagenumbering{gobble}
\onecolumn

\title{Uncertainty Quantification for Sparse Deep Learning}
\author{  Yuexi Wang and Veronika Ro\v{c}kov\'{a} 
}
\date{Booth School of Business, University of Chicago}

\maketitle

\section{Supplemental Material}
\subsection{Rudiments}
With the prior measure $\Pi(\cdot)$ on $\F(L,\p, s)$, given observed data $\Y^{(n)}=(Y_1, \ldots, Y_n)'$, inference about $f_0$ is carried out via the posterior distribution
\begin{equation*}
\Pi(A|\Y^{(n)}, \{\x_i\}_{i=1}^n)=\frac{\int_A \prod_{i=1}^n \Pi_f(Y_i|\x_i) d\Pi(f) }{\int \prod_{i=1}^n \Pi_f(Y_i|\x_i) d\Pi(f) }, \forall A \in \mB
\end{equation*}
where $\mB$ is a $\sigma$-field on $\F(L,\p, s)$ and where $\Pi_f(Y_i|\x_i)$ is the likelihood function for the output $Y_i$ under $f$.

\subsection{Posterior Concentration Rate}
{First, we show that the posterior concentrates at the optimal (near-minimax) rate. We modify the result in  \citet{polson2018posterior}  to our prior which differs in two aspects: 
(1) the top layer is fully connected, (2) the top layer coefficients are assigned a Gaussian prior.
First, we show  that our fully-connected top layer  networks can approximate $f_0$ as well as the networks considered in \citet{polson2018posterior} (i.e. with a sparse top layer). 
The following Lemma demonstrates how one can construct a fully connected top layer network from any network considered in PR18 so that their outputs are the same.
 A graphical illustration of this construction can be found in Figure \ref{fig:network_construct}.
 
 \begin{figure}[!ht]
\begin{center}
\subfigure[$f^{DL}_{\B}(\x)$ with sparse top layer]{
\includegraphics[width=0.4\textwidth]{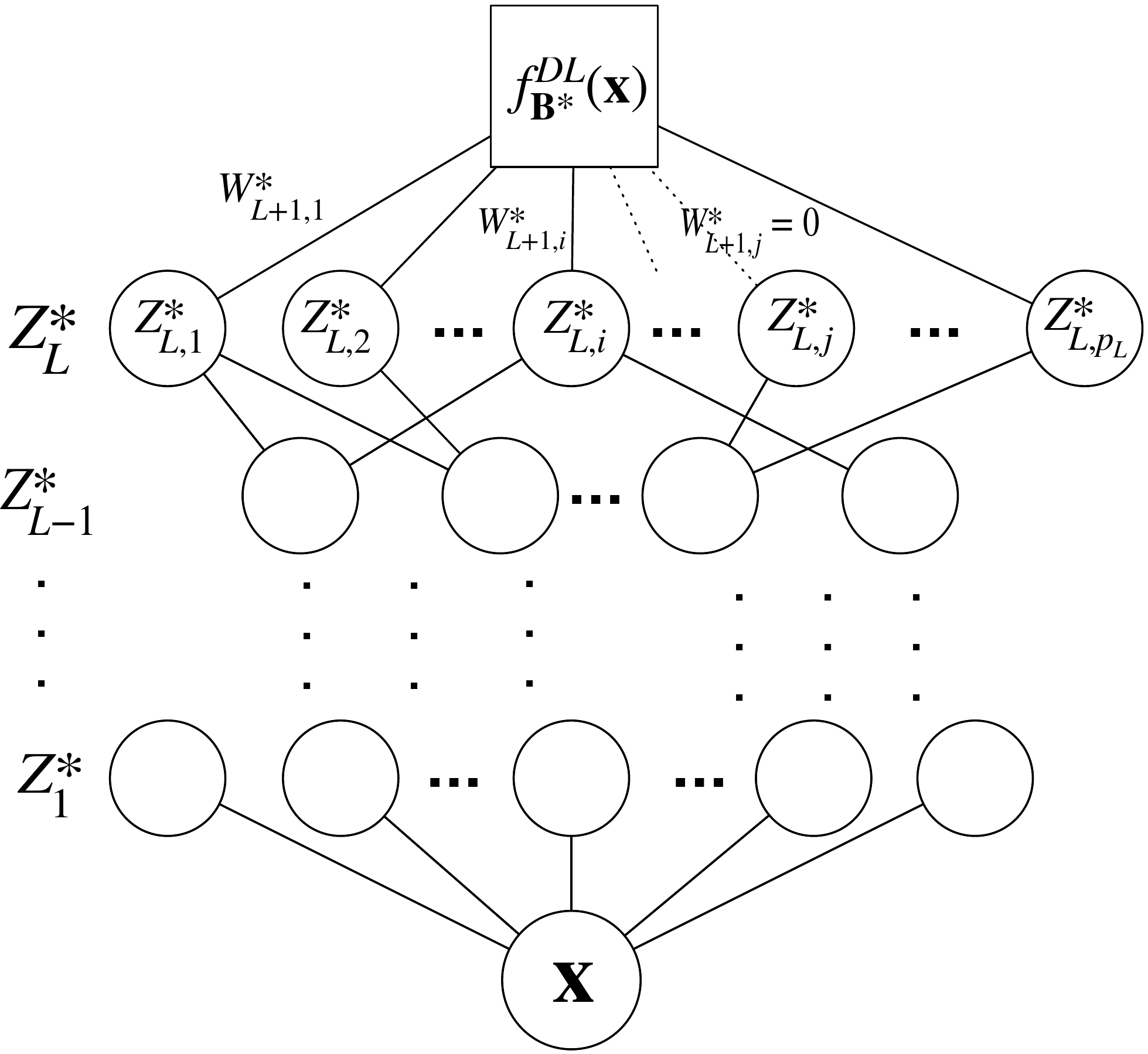}
}
\subfigure[$f^{DL}_{\B^*}(\x)$ with fully-connected top layer]{
\includegraphics[width=0.4\textwidth]{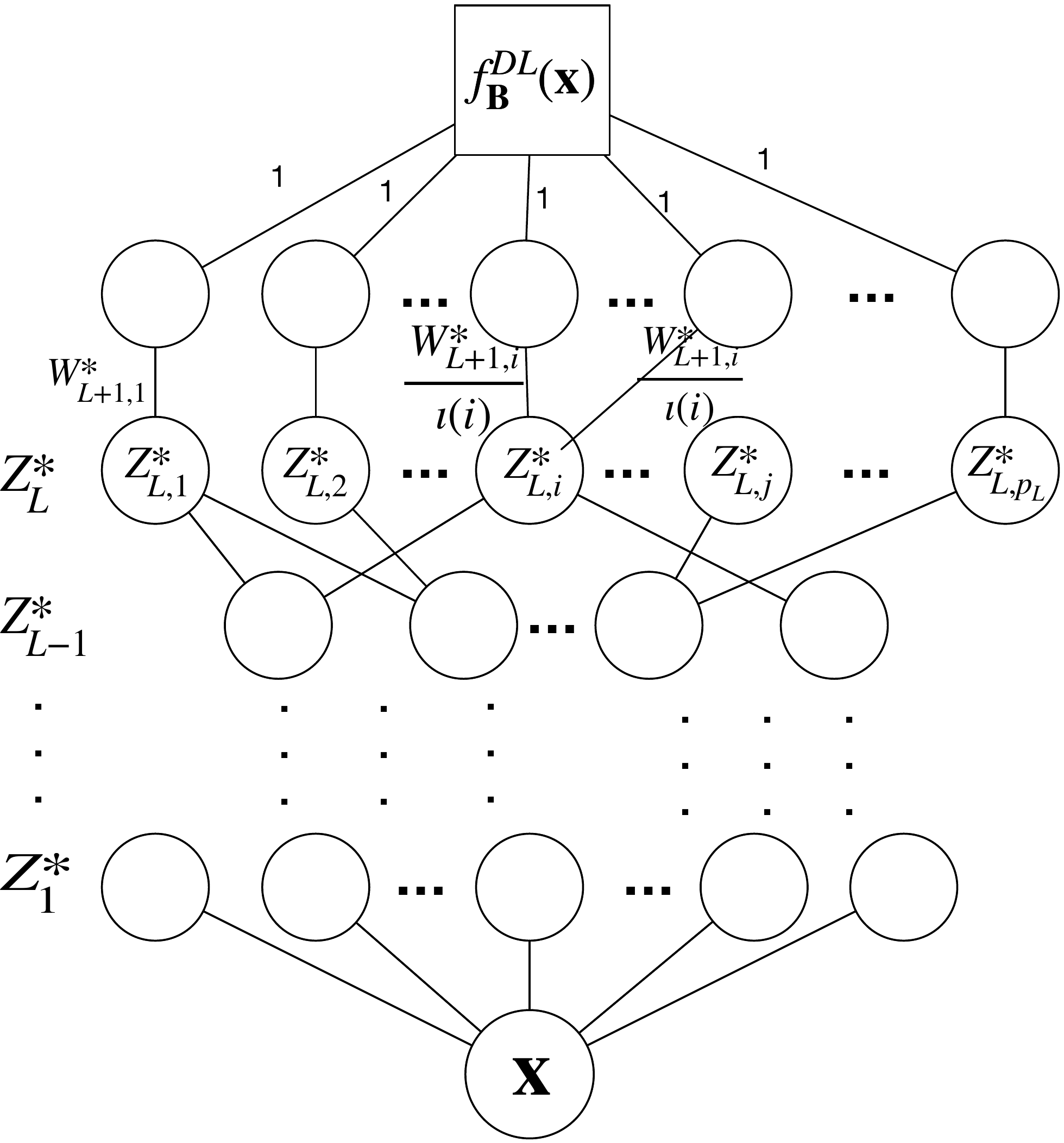}
}
\end{center}
\caption{Network Construction}\label{fig:network_construct}
\end{figure}

\begin{lemma}\label{lem:approx_error}
Assume a sparse network $f_{\B^*}^{DL} \in \tilde\F(L,\p^*,s^*)$ of the form (6) in PR18 with a sparsity pattern $\gamma$, where $\tilde \F(L,\p^*,s^*)$ is defined in Section 4 of PR18. 
With $\p^*=(p,p^*_1, \ldots, p^*_L,1)\in\N^{L+2}$ and  $\abs{\gamma}=s^*$, there exists at least one network $f_{\B}^{DL}\in \F(L+1, \p, s)$ with $\p=(p,p^*_1, \ldots,p^*_L, p^*_L,1)\in\N^{L+3}$ and  $\abs{\gamma}=s\leq s^*+2p^*_L$  such that $f_{B^*}^{DL}(\x)=f_{\B}^{DL}(\x)$ for any $\x \in \R^p$.
\end{lemma}
}
\begin{proof}
We construct one function $f_{\B}^{DL}$ that satisfies the stated conditions.  We denote $\B=\{(W_l,b_l):1\leq l\leq L+2\}$ such that $\p=(p, p^*_1, \ldots, p^*_L, p^*_L,1) \in \N^{L+3}$ and  choose the same deep coefficients $\{W_l, b_l\}=\{W^*_l, b^*_l\}$ for each $1\leq l\leq L$. The parameters of the top layer  are set as $W_{L+2}=1_{p^*_L}'$ and  $b_{L+2}=b^*_{L+1}$. 
Choosing the matrix $W_{L+1}$ in a way such that  $W_{L+1}^{\prime}1_{p^*_L}=W_{L+1}^{*\prime}$ we obtain
\[
f_{\B}^{DL}(\x)=W_{L+2}Z_{L+1}+b_{L+2}=W_{L+2}W_{L+1}Z^*_{L} +b^*_{L+1}=W^*_{L+1}Z^*_{L}+b^*_{L+1}=f_{\B*}^{DL}(\x).
\]

The procedure we use to generate $W_{L+1}$ from $W^*_{L+1}$ can be found in Algorithm \ref{alg:network_construct}.

\begin{algorithm}[!ht]
\caption{Network Construction of $\F(L+1, \p,s)$ from $\tilde \F(L,\p^*, s^*)$}\label{alg:network_construct}
\begin{algorithmic}[1]
\item We assume $W^*_{L+1,1}\neq 0$
\item Initialize $\{W_l, b_l\}_{l=1}^L=\{W_l^*, b_l^*\}_{l=1}^L, W_{L+1}=0_{p_L\times p_L},b_{L+1}=0, W_{L+2}=\1'_{p_L}, b_{L+2}=b^*_{L+1}$
\Function{ $h(j)$}{}    \Comment{the index of last connected  node (up to j) in layer L+1 of $f_{\B^*}^{DL}$}
$h(j):=\max\{ k\leq j: W_{L+1,k}\neq 0\}$
     \EndFunction
\Function{$\iota(j)$}{}\Comment{\#nodes in layer L+1 in $f_{B}^{DL}$ that will be connected to $Z_{L,h(j)}$}
\State $\iota(j) :=\sum_{i=1}^{p_L}\1(h(i)=h(j))$
\EndFunction
 \Procedure{Generate $W_{L+1}$ from $W^*_{L+1}$}{} 
 \For{each $j=1:p_{L}$}
    \If{$h(j)=j$} \Comment{when $Z_{L,j}$ previously connected in $f_{\B^*}^{DL}$}
    
    \State $W_{L+1,i,i}=\frac{W^*_{L+1,j}}{\iota(j)}$ \Comment{connect $Z_{L,j}$ to $Z_{L+1, j}$ with the averaged weights}
\Else \Comment{when $Z_{L,j}$ previously unconnected in $f_{\B^*}^{DL}$}
\State $W_{L+1,j, h(j)}= \frac{W^*_{L+1,h(j)}}{\iota(j)}$ \Comment{connect $Z_{L,h(j)}$ to $Z_{L+1, j}$ with the averaged weights}
        \EndIf

    \EndFor

\EndProcedure

\end{algorithmic}
\end{algorithm}

It turns out that the sparsity of this extended network satisfies
\[
s=s^*+\norm{W_{L+2}}_0+\norm{W_{L+1}}_0-\norm{W^*_{L+1}}_0=s^*+2p^*_L-\norm{W^*_{L+1}}_0\leq s^*+2p^*_L.\qedhere
\]

\end{proof}

With the construction from Lemma \ref{lem:approx_error}, our network class could achieve at least the same approximation error as the one in \citet{schmidt2017nonparametric}. 
To recover the posterior concentration rate results in Theorem 6.1 in PR18, we impose the following  conditions on $(L, s, N)$
\[
\left\{\begin{array}{l}
L^*\propto \log(n)\\
s^* \lesssim n^{p/(2\alpha+p)}\\
N^* \propto n^{p/(2\alpha+p)}/\log(n)
\end{array}
\right. \Rightarrow \left\{\begin{array}{l}
L=L^*+1 \propto \log(n)\\
s\leq s^*+2p^*_L=s^*+24pN^* \lesssim n^{p/(2\alpha+p)}+n^{p/(2\alpha+p)}\frac{p}{\log(n)} \lesssim n^{p/(2\alpha+p)}\\
N=N^* \propto n^{p/(2\alpha+p)}/\log(n)
\end{array}
\right.
\]
The assumptions on the network structure (depth, width and sparsity) maintain very similar for our new prior.

We formally state the posterior concentration result for our prior below.
\begin{theorem}\label{thm:post_concentrate}
Assume $f_0\in \mH^\alpha_p$ where $p=O(1)$ as $n\to \infty$, $\alpha<p$ and $\norm{f_0}_\infty\leq F$. Let $L, s$ be as in \eqref{eq:An_cond}, and $\p=(p, 12 pN,\ldots, 12pN, 1)'\in \N^{L+2}$, where $N=C_N\lfloor n^{p/(2\alpha+p)}/\log(n)\rfloor$ for some $C_N>0$. Under the priors from Section 2.1, the posterior distribution concentrates at the rate $\epsilon_n=n^{-\alpha/(2\alpha+p)}\log^\delta(n)$ for some $\delta>1$ in the sense that
\[
\Pi(f_{\B}^{DL}\in \F(L,p,s):\norm{f-f_0}_n>M_n\epsilon_n\mid Y^{(n)})\to 0
\]
in $\prob^n_0$ probability as $n\to \infty$ for any $M_n\to \infty$.
\end{theorem}
\begin{proof}
The statement can be proved as in Rockova and Polson (2018) by verifying the following three conditions (adopted from \citet{ghosal2007convergence})
\begin{align}
\sup_{\epsilon>\epsilon_n} \log \mE\left(\frac{\epsilon}{36}; A_{\epsilon,1}\cap \F_n;\norm{\cdot}_n \right)&\leq n\epsilon_n^2\label{eq:entropy}\\
\Pi(A_{\epsilon_n,1})&\geq e^{-dn\epsilon_n^2}\label{eq:prior_concentrate}\\
\Pi(\F\backslash\F_n)&\leq e^{-(d+2)n\epsilon_n^2}\quad\text{for some $d>2$}.\label{eq:prior_mass}
\end{align}

We define $\F_n$, for some $C_n=Cn^{p/(2\alpha+p)}\log^{2\delta}(n)$  and    $C>0$,  as
\[
\F_n=\{f_{\B}^{DL}\in \F(L,\p,s): \norm{W_{L+1}}^2_2+b^2_{L+1}\leq C_n\}.
\]
Here $\F_n \subset \F(L,\p,s)$ is an approximating space (a sieve) consisting of functions whose top layer weights are contained in a ball of radius $\sqrt{C_n}$ in $\R^{p_L+1}$. We show that this sieve contains most of the prior mass as required in (\ref{eq:prior_mass}) for $C>0$ large enough. Indeed, 
because $p=\mathcal{O}(1)$ and 
\[
p_L+1=12\,p\,N+1 \asymp n^{p/(2\alpha+p)}/\log(n)
\]
we have
\begin{align*}
\Pi(\F\backslash \F_n)&=\prob\left(\norm{W_{L+1}}^2_2+b^2_{L+1}>  C_n \right)\\
&=\prob(\chi_{p_L+1}^2> C_n)=\prob(e^{\frac{1}{4}\chi_{p_L+1}^2}>e^{\frac{C_n}{4}})\leq e^{-\frac{C_n}{4}}2^{(p_L+1)/2}\rightarrow 0.
\end{align*}

Next, we want to verify the entropy condition (\ref{eq:entropy}). Because
\[
\{f_{\B}^{DL}\in \F_n: \norm{f}_{\infty}\leq \epsilon\}\subset \{f_{\B}^{DL}\in \F_n: \norm{f}_n\leq \epsilon\}
\]
we  have
\begin{align*}
\sup_{\epsilon>\epsilon_n} \log \mE\left(\frac{\epsilon}{36}; f\in \F_n;\norm{\cdot}_\infty \right)&\lesssim \log\left\{ \underbrace{ \left( \frac{2}{\frac{\epsilon_n/36}{V(L+1)}}\right)^{s-(p_L+1)}}_{(I)}\underbrace{\left(\frac{\sqrt{C_n}}{\frac{\epsilon_n/36}{V(L+1)}}\right)^{p_L+1}}_{(II)}\right\}\\
&\lesssim(s+1) \log\left( \frac{72}{\epsilon_n}(L+1)(12pN+1)^{2(L+1)} \right)+(p_L+1)\log(n^{p/(2\alpha+p)}\log^{2\delta}(n))\\
&\lesssim n^{p/(2\alpha+p)}\log(n)\log\big(n/\log^\delta(n)\big)+ n^{p/(2\alpha+p)}/\log(n) \log\big(n\log(n)\big)\\
&\lesssim  n^{p/(2\alpha+p)} \log^2(n) \lesssim n\epsilon_n^2
\end{align*}
for some $\delta>1$,  where 
\begin{equation}\label{eq:V}
V=\prod_{l=0}^{L+1}(P_l+1)
\end{equation}
and using the fact that $s\lesssim n^{p/(2\alpha+p)}$ and $L\asymp \log(n)$.

The covering number $\mE(\frac{\epsilon}{36}; f\in \F_n;\norm{\cdot}_\infty )$ consists of two parts. The part (I) stands for the covering number  for the deep architecture, while the part (II) is the covering number for the top layer. The calculations of the covering numbers are derived from Lemma 12 of \citet{schmidt2017nonparametric} which shows
\[
\norm{f_{\B}^{DL}-f_{\B^*}^{DL}}_\infty \leq \norm{\B-\B^*}_\infty V(L+1)
\]
with $V$ defined as in (\ref{eq:V}). To make sure $\norm{f_{\B}^{DL}-f_{\B^*}^{DL}}_\infty\leq \frac{\epsilon_n}{36}$, we want $\norm{\B-\B^*}_\infty\leq \frac{\epsilon_n/36}{2V(L+1)}$. Since all deep parameters are bounded in absolute value by one, we can discretize the unit cube $[-1,1]^{s-p_L-1}$ with a grid of a diameter $\frac{\epsilon_n/36}{2V(L+1)}$ and obtain the covering number in part (I). For the top layer, the weights and the bias term are contained inside a  $(p_L+1)$-dimensional ball with a radius $\sqrt{C_n}$. Part(II) for $\norm{\cdot}_\infty$ is bounded by the $\frac{\epsilon_n/36}{2V(L+1)}$-covering number of a Euclidean ball of radius $\sqrt{C_n}$ in $(p_L+1)$-dimensional space \citep{edmunds2008function}.

Last, we need to show that the prior concentrates enough mass around the truth in the sense of  (\ref{eq:prior_concentrate}). From Lemma \ref{lem:approx_error} and Lemma 5.1 in PR18, we know that there exists a neural network $\hat f_{\hat \B}\in \F_n(L,\p,s)$, such that
\[
\norm{\hat f_{\hat \B}-f_0}_n\leq \epsilon/2.
\]
We denote the connectivity pattern of $\hat f_{\hat \B}$ as $\hat \gamma$ (with $\hat s=|\hat\gamma|$) and the corresponding set of coefficients as $\hat \B$. Following the same arguments as in  PR18,  we have
\[
\{f_{\B}^{DL}\in \F_n(L,\p,s):\norm{f_{\B}^{DL}-f_0}_n\leq \epsilon_n \}\supset \{f_{\B}^{DL}\in \F_n( L,\p,\hat\gamma):\norm{f_{\B}^{DL}-f_0}_n\leq \epsilon_n/2\}.
\]
We now denote with $\b\in\R^{T}$ and $\hat\b\in\R^{T}$ the vectorized nonzero coefficients in $\B$ and $\hat\B$ that have the sparsity pattern $\hat \gamma$. We use $\gamma(\b)$ to pin down the sparsity pattern of $\b$.
Using  Lemma 12 of \citet{schmidt2017nonparametric} we have
\begin{equation}
\left\{f_{\B}^{DL}\in \F_n( L,\p,\hat\gamma):\norm{f_{\B}^{DL}-f_0}_n\leq \epsilon_n/2 \right\} \supset \left\{\b\in \R^{T}: \gamma(\beta)=\hat \gamma \, \mathrm{ and } \,\norm{\b-\hat \b}_\infty\leq \frac{\epsilon_n}{2V(L+1)} \right\}.
\end{equation}
Altogether, we can write
\begin{align*}
&\Pi(f_{\B}^{DL}\in \F_n(L,\p,\hat s):\norm{f_{\B}^{DL}-f_0}_n\leq \epsilon_n)>\frac{\Pi(f_{\B}^{DL}\in \F_n( L,\p,\hat\gamma):\norm{f_{\B}^{DL}-f_0}_n\leq \epsilon_n/2)}{{{T-p_L-1}\choose{\hat s-p_L-1}}}\\
&>\frac{1}{{{T-p_L-1}\choose{\hat s-p_L-1}}} \Pi\left(\b\in \R^{T}:  \gamma(\beta)=\hat \gamma \, \mathrm{ and } \, \norm{\b-\hat \b}_\infty\leq \frac{\epsilon_n}{2V(L+1)}\right).
\end{align*}
We note that with $\hat s\asymp n^{p/(2\alpha+p)}, L\asymp \log(n)$ and $N\asymp n^{p/(2\alpha+p)}/\log(n)$
\[
\frac{1}{{{T-p_L-1}\choose{\hat s-p_L-1}}} \geq e^{-(L+1)\hat s\log(12pN)}> e^{-D_1\log^2(n)n^{p/(2\alpha+p)}}
\]
for some $D_1>0$. In addition, under the uniform prior on the deep coefficients and the standard normal prior on the top layer, we can write
\begin{align}
&\Pi\left(\b\in \R^{T}: \gamma(\beta)=\hat \gamma \, \mathrm{ and } \,  \norm{\b-\hat \b}_\infty\leq \frac{\epsilon_n}{2V(L+1)}\right)\geq \left(\frac{\epsilon_n}{2V(L+1)}\right)^{\hat s-p_L-1}\prod_{j>T-p_L-1}\Pi\left(\abs{\beta_j-\hat \beta_j}\leq \frac{\epsilon_n}{2V(L+1)}\right)\nonumber\\
&=\left(\frac{\epsilon_n}{2V(L+1)}\right)^{\hat s-p_L-1}\prod_{j>T-p_L-1} \int_{ -\frac{\epsilon_n}{2V(L+1)}}^{\frac{\epsilon_n}{2V(L+1)}}d\Pi(\beta_j
-\hat \beta_j)\label{eq:prior_mass2}.
\end{align}
where the last $T-p_L-1$ coefficients in $\b$ are the top layer weights and bias as shown in (\ref{eq:beta_prior}).

We want to recenter the normal distribution at $0$ rather than $\hat \beta_j$ by using the following inequality
\begin{align*}
\frac{dN(\hat\beta_j,1)}{dN(0,\frac{1}{2})}=e^{-\frac{1}{2}(\beta_j-\hat\beta_j)^2+\beta_j^2}=e^{\frac{1}{2}(\beta_j+\hat\beta_j)^2-\hat\beta_j^2}\geq e^{-\hat\beta_j^2}.
\end{align*}
Then we can continue with the lower bound for (\ref{eq:prior_mass2}) as follows
\begin{align*}
(\ref{eq:prior_mass2})
&\geq \left(\frac{\epsilon_n}{2V(L+1)}\right)^{\hat s-p_L-1} e^{-\sum_{j>T-p_L-1}\hat \beta_j^2} \left(\int_{ -\frac{\epsilon_n}{2V(L+1)}}^{\frac{\epsilon_n}{2V(L+1)}}dN\left(0,\frac{1}{2}\right)\right)^{p_L+1}\\ 
&\geq \left(\frac{\epsilon_n}{2V(L+1)}\right)^{\hat s-p_L-1}e^{- C_n} \left(e^{-(\frac{\epsilon_n}{2V(L+1)})^2} \frac{\epsilon_n}{\sqrt{\pi}V(L+1)} \right)^{p_L+1}\\
&\geq \left(\frac{2}{\sqrt{2\pi}}\right)^{p_{L}+1} \left(\frac{\epsilon_n}{2V(L+1)}\right)^{\hat s} e^{-C_n}e^{-\frac{(p_L+1)\epsilon_n}{4(12pN+1)^{(L+1)}(L+1)}}
\geq e^{-D_2n^{p/(2\alpha+p)}\log^2(n)}
\end{align*}
for some $ D_2>0$ and recall that $C_n= Cn^{p/(2\alpha+p)}\log^{2\delta}(n)$.Thus we can combine the bounds and conclude that $e^{-(D_1+D_2)n^{p/(2\alpha+p)}\log^2(n)}\geq e^{-dn\epsilon_n^2}$ for some $\delta>1$ and $d>D_1+D_2$. The proof is now complete.

\end{proof}

It is worth noting that the same concentration rate still holds if we use $N(0,1)$ prior on {\em all} parameters. We could define 
\[\F_n=\{\norm{\beta}_2^2\leq C_n\}.\]

The prior mass condition in (\ref{eq:prior_mass}) is
\[
\Pi(\F\backslash \F_n)=\prob(\chi_s^2>C_n)\leq e^{-C_1 n^{p/(2\alpha+p)}\log^{2\delta}(n)}.
\]
The entropy condition in (\ref{eq:entropy}) is
\begin{align*}
\sup_{\epsilon>\epsilon_n} \log \mE(\frac{\epsilon}{36}, f\in \F_n;\norm{\cdot}_\infty )&\lesssim \log\left\{\left(\frac{\sqrt{C_n}}{\frac{\epsilon_n/36}{V(L+1)}}\right)^{s}\right\}\\
&\lesssim(s+1) \log\left( \frac{72}{\epsilon_n}(L+1)(12pN+1)^{2(L+1)} \right)+s\log(Cn^{p/(2\alpha+p)}\log^{2\delta}(n))\\
&\lesssim n^{p/(2\alpha+p)}\log(n)\log\big(n/\log^\delta(n)\big)+ n^{p/(2\alpha+p)}\log(n\log (n))\\
&\lesssim n\epsilon_n^2
\end{align*}
for some $\delta>1$, using the fact that $s\lesssim n^{p/(2\alpha+p)}$ and $L\asymp \log(n)$.

The prior concentration condition in (\ref{eq:prior_concentrate}) can be proved by changing (\ref{eq:prior_mass2}) into
\begin{align*}
&\Pi(\beta\in \R^{T}: \gamma(\beta)=\hat \gamma, \sum_{j}\beta_j^2\leq C_n\text{ and } \norm{\beta-\hat \beta}_\infty\leq \frac{\epsilon_n}{2V(L+1)})\\
&\geq  e^{-\sum_{j}\hat \beta_j^2} \left(\int_{ -\frac{\epsilon_n}{2V(L+1)}}^{\frac{\epsilon_n}{2V(L+1)}}dN(0,\frac{1}{2}) \right)^{\hat s}\\ 
&\geq e^{-C_n} \left(e^{-(\frac{\epsilon_n}{2V(L+1)})^2} \frac{\epsilon_n}{\sqrt{\pi}V(L+1)} \right)^{\hat s}\\
&\geq e^{-C_n} \left(\frac{\epsilon_n}{\sqrt{\pi}V(L+1)}\right)^{\hat s} e^{-\frac{\hat s\epsilon_n}{4(12pN+1)^{(L+1)}(L+1)}}
\geq e^{-Dn^{p/(2\alpha+p)}\log^2(n)}\qedhere.
\end{align*}

\begin{theorem}(adaptive priors)\label{thm:adaptive}
Assume $f_0 \in \mH^\alpha_p$, where $p=O(1)$ as $n\to \infty$, $\alpha<p$, and $\norm{f_0}_\infty\leq F$. Let $L \asymp \log(n)$ and assume priors for N and s as in (\ref{eq:N_prior}) and (\ref{eq:s_prior}). Assume the prior of $f$ as  given by (\ref{eq:beta}) and (\ref{eq:gamma}). Then the posterior distribution concentrates at the rate $\xi_n=n^{-\alpha/(2\alpha+p)}\log^\delta(n)$ for $\delta>1$ in the sense that
\[
\Pi(f\in \F(L):\norm{f-f_0}_L>M_n\xi_n\mid \Y^{(n)})\to 0
\]
in $\P_0^n$ probability as $n\to\infty$ for any $M_n\to \infty$.
\end{theorem}

The proof for Theorem \ref{thm:adaptive} follows the same techniques used in Theorem 6.2 of PR18. And this adaptive results also hold for networks with standard normal priors on all weights.

\subsection{Preparations for Main Theorems}
The general framework for first-order approximation of functionals is as follows
\begin{theorem}\label{thm:first_order} \citep{castillo2015bernstein} Consider the  model $\P^n_0 $, a real-valued functional $f\to \Psi(f)$ and $\expect{\cdot, \cdot}_L,\Psi_0^{(1)}, W_n, $ as defined above. Suppose that (\ref{eq:rn}) is satisfied, and denote
\[
\hat\Psi= \Psi(f_0)+\frac{W_n(\Psi_0^{(1)})}{\sqrt n}, \qquad V_{0}=\norm{\Psi_0^{(1)}}_L^2.
\]
Let $\Pi$ be a prior distribution on $f$. Let $A_n$ be any measurable set such that 
\[
\Pi(A_n\mid \Y^{(n)})=1+o_P(1), \, \text{ as } n\to \infty.
\]
Then for any real $t$ with $f_t$ as 
\begin{equation*}
f_t=f-\frac{t\Psi_0^{(1)}}{\sqrt n},
\end{equation*}
we could write
\begin{equation*}
\E^\Pi[e^{t\sqrt n (\Psi(f)-\hat\Psi)}\mid  \Y^{(n)}, A_n]=e^{o_P(1)+t^2 V_{0}/2}\frac{\int_{A_n}e^{\ell_n(f_t)-\ell_n(f_0)d\Pi(f)}}{\int_{A_n}e^{\ell_n(f)-\ell_n(f_0)d\Pi(f)}}.
\end{equation*}
Moreover, if
\begin{equation}\label{eq:ft_ratio}
\frac{\int_{A_n}e^{\ell_n(f_t)-\ell_n(f_0)d\Pi(f)}}{\int_{A_n}e^{\ell_n(f)-\ell_n(f_0)d\Pi(f)}}=1+o_P(1), \, \forall t\in \R
\end{equation}
is satisfied, then the posterior distribution of $\sqrt n(\Psi(f)-\hat \Psi)$ is asymptotically normal and mean-zero, with variance $V_0$.
\end{theorem}
\begin{proof}
Set $R_n(\cdot, \cdot)=0, \Psi_0^{(2)}=0,\mu_n=0$ in Theorem 2.1 of \citet{castillo2015bernstein}.
\end{proof}

\paragraph{Projection of Functions} The intuition of our projection conditional on $(\gamma, Z)$ is to maintain the same partitions for the shifted function in (\ref{eq:c_o_m}) and perform the \textit{change of measure} locally.  We first give the notation for $Z^L$, which are the nodes in the top layer. Let $Z_{Lj}, j=1, \ldots, p_L$ denote the $j^{th}$ node in $L^{th}$ layer, which can be written as a sum of local linear functions, respectively:
\[
Z_{Lj}(\x)=\sum_{k=1}^{K_L} \1(\x \in \Omega_{k}^j)\{ \tilde \beta_k^{j'} \x+\tilde\alpha_k^j \}
\]
here the partitions $\{\Omega_k^j\}_{k=1}^{K_L}$ and coefficients $\{\tilde\beta_k^j, \tilde\alpha_k^j\}_{k=1}^{K_L}$ are determined by $\{W_l, b_l\}_{l=1}^L$.

For simplicity of notation, we denote $W_{L+1}=(w_1, \ldots, w_{p_L})'$. Then the output can be written as:
\begin{align*}
\displaystyle
f(\x)&=\sum_{j=1}^{p_L} w_j Z_{Lj}(\x)+b_{L+1}\\
&=\sum_{k_1=1}^{K_L}\cdots\sum_{k_{p_L}=1}^{K_L}\1\left(\x \in \bigcap_{j=1}^{p_L} \Omega_{k_j}^j\right) \left\{ \left(\sum_{j=1}^{p_L} w_j \tilde\beta_{k_j}^{j'} \right) \x +\left(\sum_{j=1}^{p_L} w_j \tilde \alpha_{k_j}^j+b_{L+1}\right)  \right\}.
\end{align*}

We denote the projection of function $a(\x)$ conditional on $\{W_l, b_l\}_{l=1}^L$ with $a^\gamma_{[Z]}$, since conditional on $\{W_l, b_l\}_{l=1}^L$ is equivalent to conditional on $(\gamma, Z)$:
\begin{align*}
(W^a, b^a)&={\arg\min}_{W_{L+1}, b_{L+1}\in \F_n(L,\p,\gamma,Z)}\norm{ WZ_L(\x)+b - a(\x)}_L,\\
a^\gamma_{[Z]} (\x)&=W^aZ_L(\x)+b^a.
\end{align*}
The projection $a^\gamma_{[Z]}$ can also be viewed as the best approximation to $a$ conditional on $(\gamma, Z)$.

Similarly, we denote projection of $f_0$ onto $\{W_l, b_l\}_{l=1}^L$ as $f^\gamma_{0[Z]}$:
\begin{align}
(W^0, b^0)&={\arg\min}_{W_{L+1}, b_{L+1}\in \F_n(L,\p,\gamma,Z)} \norm{ WZ_L(\x)+b - f_0(\x)}_L,\label{eq:W0b0}\\
f^\gamma_{0[Z]}(\x)&=W^0Z_L(\x)+b^0\label{eq:f_proj}.
\end{align}
Note that $f \in \{WZ_L(\x)+b: W \in \R^{p_L}, b\in \R\}$, so naturally we have $\norm{f^\gamma_{0[Z]}-f_0}_L\leq  \norm{f-f_0}_L$.

\subsection{Proof of Theorem \ref{thm:linear}}\label{thm:proof_linear}
We will perform the analysis locally on the sets $A_n\equiv A_n^{M_n}$ from \eqref{eq:An} for some $M_n\rightarrow\infty$. We use the fact that convergence of Laplace transforms for all $t$ in probability implies convergence in distribution in probability \citep{castillo2015bernstein}.
The posterior decomposes into a mixture of laws with weights $\Pi(\gamma\mid  \Y^{(n)})$, where $\gamma$ is the vector encoding the connectivity pattern with prior in (\ref{eq:gamma_prior}).
We denote  with $I_{n,\gamma}= \E^\Pi[e^{t\sqrt n (\Psi(f)-\hat\Psi)}\mid \Y^{(n)}, A_{n},\gamma]$ and write
\begin{align*}
I_{n}:&=\E^\Pi[e^{t\sqrt n (\Psi(f)-\hat\Psi)}\mid \Y^{(n)}, A_{n}]
= \sum_{\gamma \in \mV^{\p,s}}  \Pi(\gamma\mid \Y^{(n)}, A_{n}) I_{n,\gamma}.
\end{align*}
Next, we want to show that on the event $A_n$ and uniformly for all $ \gamma  \in \mV^{\p,s}$
\[
I_{n,\gamma} =e^{o_P(1)+t^2V_0/2}(1+o(1))\quad \text{ as } n\to\infty
\]
so that $I_n=e^{o_P(1)+t^2V_0/2}(1+o(1))$.

We choose $\gamma$ such that $ \F(L,\p,\gamma)\cap A_n\neq\emptyset$ and for $f\in \F(L,\p,\gamma)\cap A_n$ we  expand the linear functional  as
$\Psi(f)-\Psi(f_0)=\expect{a, f-f_0}_L$ which yields
\begin{align*}
\Psi_{0}^{(1)}&=a,\\
r(f,f_0)&=0.
\end{align*}
The remainder condition \eqref{eq:rn} is thus trivially satisfied.
To verify the second condition \eqref{eq:c_o_m}, we choose the shifted function $f_t$ as
\[
f_t=f-\frac{t a}{\sqrt n}.
\]
Due to the fact that our  class of neural networks has a top linear layer, the function $f_t$ shares the same deep connectivity structure as $f$ where only the top layer intercepts $b_{L+1}^t$ have been shifted.
The {\em change of measure} thus only influences $b_{L+1}$ where $b^t_{L+1}=b_{L+1}-\frac{t a}{\sqrt n}$. Next, we can write
\begin{align}\label{eq:I_n}
I_{n,\gamma}&=e^{\frac{t^2}{2}\norm{a}^2_L} \times \frac{\int_{A_n}e^{\ell_n(f_t)-\ell_n(f_0)}d\Pi(f\mid \gamma)}{ \int_{A_n}e^{\ell_n(f)-\ell_n(f_0)}d\Pi(f\mid \gamma)}\\
&=e^{\frac{t^2}{2}\norm{a}^2_L} \times \frac{\int_{f_t+ \frac{t a}{\sqrt n}  \in A_n}e^{\ell_n(f_t)-\ell_n(f_0)}d\Pi(f_t\mid \gamma)\frac{d\Pi(f\mid \gamma)}{d\Pi(f_t\mid \gamma)}}{ \int_{A_n}e^{\ell_n(f)-\ell_n(f_0)}d\Pi(f\mid \gamma)}.
\end{align}

Next, we show that the ratio above converges to $1$ as $n\to \infty$. We have
 \begin{align*}
 &\frac{d\Pi(f\mid \gamma)}{d\Pi(f_t\mid\gamma)}=\frac{ d\Pi(\{W_l,b_l\}_{l=1}^L, W_{L+1}, b_{L+1}\mid\gamma ) }{d\Pi(\{W_l,b_l\}_{l=1}^L, W_{L+1}, b^t_{L+1}\mid\gamma)}=\frac{ d\Pi(\{W_l,b_l\}_{l=1}^L\mid\gamma )d\Pi(W_{L+1})d\Pi( b_{L+1}) }{d\Pi(\{W_l,b_l\}_{l=1}^L\mid\gamma )\Pi(W_{L+1})\Pi( b^t_{L+1})}=\frac{d\Pi( b_{L+1})}{d\Pi( b^t_{L+1})}\\
& \frac{d\Pi( b_{L+1})}{d\Pi( b^t_{L+1})} =\frac{\phi( b_{L+1})}{\phi( b_{L+1}-\frac{ta}{\sqrt n})}=\exp \left\{-\frac{1}{2}\left[b_{L+1}^2-( b_{L+1}-\frac{ta}{\sqrt n})^2\right]\right\}=
\exp\left(-\frac{at b_{L+1}}{\sqrt n}+\frac{t^2a^2}{2n}\right)
 \end{align*}
Next, we note  (from the definition of the sieve $\mathcal F_n$ and $C_n$ in the proof of Theorem \ref{thm:post_concentrate}) 
 \[\frac{|b_{L+1}|}{\sqrt n}\leq \frac{\sqrt{C_n}}{\sqrt n}\lesssim n^{-\frac{\alpha}{2\alpha+p}\log^{\delta}(n)}.\]

Going back to \eqref{eq:I_n}, we now have for some $c>0$
\begin{multline}\label{eq:sandwich}
e^{-c\,n^{-\frac{\alpha}{2\alpha+p}}\log^{\delta}(n)+\frac{t^2a^2}{2n} + \frac{t^2}{2}\norm{a}^2_L}\times\frac{\Pi\left({f+\frac{ta}{\sqrt n}\in A_n } \mid \Y^{(n)},\gamma\right)}{\Pi\left({f\in A_n} \mid \Y^{(n)},\gamma\right)}
\leq I_{n,\gamma}\\
\leq e^{c\,n^{-\frac{\alpha}{2\alpha+p}}\log^{\delta}(n)+\frac{t^2a^2}{2n}+\frac{t^2}{2}\norm{a}^2_L}\times\frac{\Pi\left({f+\frac{ta}{\sqrt n}\in A_n } \mid \Y^{(n)},\gamma\right)}{\Pi\left({f\in A_n} \mid \Y^{(n)},\gamma\right)}.
\end{multline}

Next, from 
\[
 \norm{f -f_0}_L - \norm{\frac{t a}{\sqrt n}}_L \leq \norm{f+\frac{t a}{\sqrt n} -f_0}_L \leq \norm{f -f_0}_L + \norm{\frac{t a}{\sqrt n}}_L
\]
it is clear that 
\[
\left\{f: \norm{f -f_0}_L \leq M_n\xi_n-\norm{\frac{t a}{\sqrt n}}_L\right\}\subset \left\{f: \norm{f+\frac{t a}{\sqrt n} -f_0}_L\leq M_n\xi_n\right\} \subset \left\{f: \norm{f -f_0}_L \leq M_n\xi_n+\norm{\frac{t a}{\sqrt n}}_L\right\}
\]
This yields 
\begin{multline*}
\Pi\left(f:\norm{f -f_0}_L \leq \xi_n-\norm{\frac{t a}{\sqrt n}}_L \mid \Y^{(n)},\gamma\right)\leq \Pi\left(f:{f+\frac{ta}{\sqrt n}\in A_n} \mid \Y^{(n)},\gamma\right)\\ \leq \Pi\left(f:\norm{f -f_0}_L \leq \xi_n+\norm{\frac{t a}{\sqrt n}}_L \mid \Y^{(n)},\gamma\right).
\end{multline*}
Since the concentration rate is slower than $1/\sqrt{n}$, i.e.  $\xi_n=n^{-\alpha/(2\alpha+p)}\log^\delta (n)\gtrsim n^{-1/2}$, 
 we have $\Pi(f+ \frac{t a}{\sqrt n}  \in A_n) \to \Pi(f\in A_n)$, as $n\to \infty$.  From the sandwich inequality \eqref{eq:sandwich}, we have $I_{n,\gamma}\rightarrow e^{\frac{t^2\norm{a}_L^2}{2}}$  for any $t\in\R$ as $n\rightarrow\infty$.

\subsection{Proof of Theorem \ref{thm:L2}}\label{thm:proof_L2}

Similar to the linear functional case, the posterior decomposes into a mixture of laws with weights $\Pi(\gamma\mid  \Y^{(n)})$, where $\gamma$ is the vector encoding the connectivity pattern with a prior in (\ref{eq:gamma_prior}). We can write
\begin{align}
I_{n}:&=\E^\Pi[e^{t\sqrt n (\Psi(f)-\hat\Psi)}\mid \Y^{(n)}, A_{n}]=\sum_{\gamma \in \mV^{\p,s}} \Pi(\gamma\mid \Y^{(n)}, A_{n}) I_{n,\gamma} \label{eq:In}
\end{align}
where
\begin{align*}
I_{n,\gamma} &:= \E^\Pi[e^{t\sqrt n (\Psi(f)-\hat\Psi)}\mid \Y^{(n)}, A_{n},\gamma].\\
\end{align*}
We further decompose each $I_{n,\gamma}$ by conditioning on the deep weights  $\{W_l, b_l\}_{l=1}^L$.
We can write
\begin{align*}
\Pi(\{W_l, b_l\}_{l=1}^{L+1}\mid \Y^{(n)} , A_n,\gamma) & =\Pi(W_{L+1}, b_{L+1}\mid \{W_l, b_l\}_{l=1}^{L} ,\Y^{(n)} , A_n,\gamma) \Pi( \{W_l, b_l\}_{l=1}^{L} \mid \Y^{(n)} , A_n,\gamma)\\
& = \Pi(W_{L+1}, b_{L+1}\mid \Y^{(n)} , A_n,\gamma, Z) \Pi( Z \mid \Y^{(n)} , A_n,\gamma),
\end{align*}
since $Z=\{Z_l\}_{l=1}^L$ is fully determined by $\{W_l, b_l\}_{l=1}^{L}$ and we can thereby replace conditioning on $\{W_l, b_l\}_{l=1}^{L}$ by conditioning on $Z$. 
We can further dissect $I_{n,\gamma}$ by conditioning on $Z$ 
\begin{align*}
I_{n,\gamma} =\int I_{n,\gamma}^Z d \Pi( Z \mid \Y^{(n)} , A_n,\gamma),\quad\text{where}\quad
I_{n,\gamma}^Z:= \int e^{t\sqrt n (\Psi(f)-\hat\Psi)} d \Pi(W_{L+1}, b_{L+1}\mid \Y^{(n)} , A_n,\gamma, Z ).
\end{align*}

In the rest of the proof, we show that $I_{n,\gamma}^Z \to \exp(-t^2V_0/2)$ uniformly for all $\gamma$ and $Z$ such that $f\in A_n$.  This can be done in two steps. First, we show that conditional on $(\Y^{(n)}, A_{n},\gamma,Z)$, $\Psi(f)$ asymptotically centers at a local $(\gamma, Z)$-dependent centering point $\hat \Psi^\gamma_Z$ with a local $(\gamma, Z)$-dependent variance $V^\gamma_Z$ (both defined later). In the second step, we show that the local centering points $\hat \Psi^\gamma_Z$  are close to  the global centering point $\hat \Psi$  and  that the local variances $V^\gamma_Z$ converge to $ V_0$ uniformly for all $\gamma$ and $Z$ such that $f\in A_n$.

We define the $(\gamma, Z)$-dependent local centering point and variance  as 
\begin{equation}\label{eq:quad_local}
\hat \Psi^\gamma_Z=\Psi (f_0)+\frac{W_n(2f^\gamma_{0[Z]})}{\sqrt n} \,\quad\text{and}\quad V^\gamma_Z=4\norm{f^\gamma_{0[Z]}}_L^2,
\end{equation}
where $f^\gamma_{0[Z]}$ is the $\|\cdot\|_L$ projection of $f_0$ on the set of deep learning networks $f$ with a connectivity pattern $\gamma$ and hidden nodes $Z$  defined in (\ref{eq:f_proj}).

For any $f\in \F(L,\p, \gamma)$, the squared $L^2$-norm functional can be expanded as
\begin{align*}
\Psi(f)-\Psi(f_0)&=2\expect{f_0, f-f_0}_L+\norm{f-f_0}^2_L\\
&=2\expect{f^\gamma_{0[Z]}, f-f_0}_L+\norm{f-f_0}^2_L+2\expect{f_0-f^\gamma_{0[Z]}, f-f_0}_L.
\end{align*} 
Note that   $\norm{f^\gamma_{0[Z]}-f_0}_L\leq \norm{f-f_0}_L$ for any $f$ which has a connectivity pattern $\gamma$ and hidden nodes $Z$.

This expansion yields the first-order and remainder terms 
\begin{align*}
\Psi_{0}^{(1)}&=2f^\gamma_{0[Z]},\\
r(f,f_0)&=\norm{f-f_0}^2_L+2\expect{f_0-f^\gamma_{0[Z]}, f-f_0}_L.
\end{align*}

To ensure asymptotical normality of $\Psi(f)$, we first need to ensure the local shape condition in (\ref{eq:rn}).
Assuming that the  smoothness $\alpha$ satisfies
\begin{equation}\label{eq:alpha_cond}
\alpha>p/2
\end{equation}
we have for $f\in A_n$ with a connectivity $\gamma$ and hidden nodes $Z$
\begin{align*}
r(f,f_0)&=\norm{f-f_0}^2_L+2\expect{f_0-f^\gamma_{0[Z]}, f-f_0}_L\\
& \leq 2\norm{f-f_0}^2_L+\norm{f_0-f^\gamma_{0[Z]}}^2_L\\
& \leq 3\norm{f-f_0}^2_L\lesssim \xi_n^2=n^{-\frac{2\alpha}{2\alpha+p}}\log^{2\delta} =o\left(\frac{1}{\sqrt n}\right).
\end{align*}

Next, to verify the second sufficient condition \eqref{eq:c_o_m} we define the shifted function $f_t$ as
\[
f_t=f-\frac{2tf^\gamma_{0[Z]}}{\sqrt n}.
\]

Then we use the  local centering point $\hat \Psi_{Z}^\gamma$ in \eqref{eq:quad_local} to define
\begin{align}
\tilde I_{n,\gamma}^Z:=&\E^\Pi[e^{t\sqrt n (\Psi(f)-\hat\Psi^\gamma_Z)}\mid \Y^{(n)}, A_{n},\gamma, Z]\label{eq:In_tilde}\\
=&e^{2{t^2}\norm{f^\gamma_{0[Z]}}^2_L} \times \frac{\int_{A_n}e^{\ell_n(f_t)-\ell_n(f_0)}d\Pi(f\mid \gamma,Z)}{ \int_{A_n}e^{\ell_n(f)-\ell_n(f_0)}d\Pi(f\mid \gamma, Z)}\nonumber\\
=&e^{2{t^2}\norm{f^\gamma_{0[Z]}}^2_L} \times \frac{\int_{f_t+\frac{2t f^\gamma_{0[Z]}}{\sqrt n} \in A_n}e^{\ell_n(f_t)-\ell_n(f_0)}d\Pi(f_t\mid \gamma,Z)\frac{d\Pi(f\mid \gamma, Z)}{d\Pi(f_t\mid \gamma, Z)}} { \int_{A_n}e^{\ell_n(f)-\ell_n(f_0)}d\Pi(f\mid \gamma, Z)}\nonumber
\end{align}

{
 For simplicity of notation, we first denote $\zeta=(W_{L+1}, b_{L+1})'\in\R^{p_L+1}$ and  $\zeta^t=(W_{L+1}^t, b_{L+1}^t)'\in \R^{p_L+1}$ and $\Delta=(W^0, b^0)'$ as defined  in (\ref{eq:W0b0}).
 Then we can simply write $\zeta^t=\zeta-\frac{2t}{\sqrt n}\Delta$.

Since all parameters are  a-priori independent and there is no sparsity structure placed on $\{W_{L+1}, b_{L+1}\}$, the prior ratio   $\frac{d\Pi(f\mid \gamma,Z)}{d\Pi(f_t\mid \gamma,Z)}$ can be calculated as
\begin{align*}
\frac{d\Pi(f\mid \gamma, Z)}{d\Pi(f_t\mid \gamma, Z)}&=\frac{d\Pi(W_{L+1})}{d\Pi(W_{L+1}^t)}\frac{d\Pi(b_{L+1})}{d\Pi(b_{L+1}^t)}=\frac{d\Pi(\zeta)}{d\Pi(\zeta^t)}\\
&=\prod_{i=1}^{p_{L+1}}\exp\left\{-\frac{1}{2}\left[\zeta^2-(\zeta_i-\frac{2t}{\sqrt n}\Delta_i)^2\right]\right\}\\
&=\exp\left\{\sum_{i=1}^{p_{L+1}} \left[-\zeta_i\frac{\Delta_i t}{\sqrt n} + \frac{2t^2\Delta_i^2}{n} \right]\right\}.
\end{align*}

Similar to our previous proof, we have under the assumption $\alpha>p/2$ 
\begin{equation}\label{eq:c_o_m_quad}
\abs{\sum_{i=1}^{p_{L+1}} \zeta_i\frac{\Delta_i t}{\sqrt n}}\leq \frac{t}{\sqrt n}\norm{\zeta}_2\norm{\Delta}_2\lesssim \frac{C_n}{\sqrt n}=o(1),
\end{equation}
where we used the fact that  both $f$ and $f_{0[Z]}^\gamma$ are contained in $A_n$  and thereby have their top coefficients contained in a ball of radius  $\sqrt{C_n}$
  (recall the definition of $C_n$ in the proof of Theorem \ref{thm:post_concentrate}).


Now, using the fact that \begin{align*}
\norm{f-f_0}_L-2\norm{\frac{t f^\gamma_{0[Z]}}{\sqrt n}}_L\leq \norm{f+ \frac{2t f^\gamma_{0[Z]}}{\sqrt n}-f_0}_L\leq \norm{f-f_0}_L+2\norm{\frac{t f^\gamma_{0[Z]}}{\sqrt n}}_L
\end{align*}
we have
\begin{multline*}
\Pi\left(f: \norm{f -f_0}_L \leq \xi_n-2\norm{\frac{t f^\gamma_{0[Z]}}{\sqrt n}}_L\mid \Y^{(n)},\gamma, Z \right) \\ \leq \Pi\left({f+\frac{2t f^\gamma_{0[Z]}}{\sqrt n} \in A_n} \mid \Y^{(n)},\gamma, Z \right) \leq \Pi\left(f: \norm{f -f_0}_L  \leq \xi_n+2\norm{\frac{t f^\gamma_{0[Z]}}{\sqrt n}}_L\mid \Y^{(n)},\gamma, Z \right).
\end{multline*}

Again, since the concentration rate is slower than $1/\sqrt{n}$, i.e.  $\xi_n=n^{-\alpha/(2\alpha+p)}\log^\delta (n)\gtrsim n^{-1/2}$, we have
\begin{equation}\label{eq:quad_set_ratio}
\frac{\Pi(f+ \frac{2t f^\gamma_{0[Z]}}{\sqrt n} \in A_n\mid \Y^{(n)},\gamma, Z)}{ \Pi( A_n\mid \Y^{(n)},\gamma, Z)} \to 1,\, \forall t\in\R.
\end{equation}
Hence, with (\ref{eq:alpha_cond}), (\ref{eq:c_o_m_quad}) and (\ref{eq:quad_set_ratio}), one  concludes $\tilde I^Z_{n,\gamma}\to e^{2t^2\norm{f^\gamma_{0[Z]}}^2_L}$ as $n\to\infty$ using a similar sandwich inequality in (\ref{eq:sandwich}).} In other words, we have
\begin{equation}\label{eq:quad_local_In}
\tilde I_{n,\gamma}^Z=e^{t^2 V^\gamma_Z/2}(1+o(1)).
\end{equation}

Recall the definition of a local centering point $\hat \Psi_Z^\gamma$ and a local variance $V_Z^\gamma$ in (\ref{eq:quad_local}). Then we can write
\begin{align*}
I_{n,\gamma}^Z&=\E^\Pi[e^{t\sqrt n (\Psi(f)-\hat \Psi)}\mid \Y^{(n)}, A_n,\gamma, Z] \\
&=\E^\Pi[e^{t\sqrt n [(\Psi(f)-\hat \Psi_Z^\gamma)+ (\hat \Psi_Z^\gamma - \hat \Psi)]}\mid \Y^{(n)}, A_n,\gamma, Z] \\
&=\tilde I_{n,\gamma}^Z\times e^{t\sqrt n(\hat \Psi^\gamma_Z-\hat\Psi)}\\
&=(1+o(1)) e^{t^2 V^\gamma_Z/2+t\sqrt n(\hat \Psi^\gamma_Z-\hat\Psi)}\\
&=(1+o(1)) e^{t^2 V_0/2 + t^2(V^\gamma_Z-V_0)/2+t\sqrt n(\hat \Psi^\gamma_Z-\hat\Psi)}.
\end{align*}
The proof will be complete once we show the following condition uniformly for all $\gamma$ such that $f\in A_n$
\begin{align*}
I_{n,\gamma}&=\int I_{n,\gamma}^Z d\Pi(Z \mid \Y^{(n)} , A_n,\gamma)  \\
&=(1+o(1)) e^{t^2V_0/2} \int   e^{t^2(V^\gamma_Z-V_0)/2+t\sqrt n(\hat \Psi^\gamma_Z-\hat\Psi)} d\Pi(Z \mid \Y^{(n)} , A_n,\gamma) \to e^{t^2V_0/2}, \text{ as } n\to\infty.
\end{align*}

This is equivalent to showing
\begin{align}\label{eq:local_converge}
 \int  e^{t^2(V^\gamma_Z-V_0)/2+t\sqrt n(\hat \Psi^\gamma_Z-\hat\Psi)} d\Pi(Z \mid \Y^{(n)} , A_n,\gamma) = 1+o_P(1).
 \end{align}

Since we work conditionally on the set $A_n$, we have $\|f^\gamma_{0[Z]}-f_0\|_L\lesssim \xi_n$ and thereby
\begin{align*} 
\sqrt n (\hat \Psi- \hat\Psi^\gamma_Z )&=W_n(f^\gamma_{0[Z]}-f_0)=o_P(1),\\
\abs{V^\gamma_z-V}&=4\abs{\norm{f^\gamma_{0[Z]}}_L^2-\norm{f_0}_L^2} \\
&\lesssim  2\norm{f_0}_L\norm{f^\gamma_{0[Z]}-f_0}_L+ \norm{f^\gamma_{0[Z]}-f_0}_L^2\\
&\lesssim \norm{f^\gamma_{0[Z]}-f_0}_L\leq \xi_n
\end{align*}
under the assumption that $\norm{f_0}_L\leq F$.

Using the smoothness assumption (\ref{eq:alpha_cond}), we have $\xi_n^2=o\left(\frac{1}{\sqrt n}\right)$. We can bound the integral  in \eqref{eq:local_converge} using the uniform bounds on $\sqrt n (\hat \Psi- \hat\Psi^\gamma_Z )$ and $\abs{V^\gamma_z-V}$ as
\begin{align*}
(\ref{eq:local_converge})=& \int e^{t^2 \xi_n/2 + t\times o_P(1) } d\Pi(Z \mid \Y^{(n)} , A_n,\gamma) \\
=& e^{t^2 \xi_n/2 + t\times o_P(1) }=e^{o_P(1)}=1+o_P(1).
 \end{align*}

Putting the pieces together, we write $I_n$ from \eqref{eq:In} as
\[
I_n= \sum_{\gamma \in \mV^{\p,s}} \Pi(\gamma\mid \Y^{(n)}, A_{n}) I_{n,\gamma}=\sum_{\gamma \in \mV^{\p, \gamma}}  \Pi(\gamma\mid \Y^{(n)}, A_{n})e^{t^2V_0/2}(1+o_P(1))=e^{t^2V_0/2}(1+o_P(1))
\]
which completes the proof.\qed

\subsection{Proof of Theorem  \ref{thm:adaptive_bvm}}\label{thm:proof_adaptive}
For our proof for Theorem \ref{thm:adaptive_bvm}, the analysis is locally conducted on the set 
\begin{equation}\label{eq:An2}
A_n^M=\{f\in \F(L):\norm{f-f_0}_L\leq M_n\xi_n\}
\end{equation}
with $\xi_n=n^{-\alpha/(2\alpha+p)}\log^\delta(n)$ for some $M>0$ and $\delta>0$. And from the results in Theorem \ref{thm:adaptive}, we know $\Pi(A_n^M\mid Y^{(n)})=1+o_p(1)$ for any $M_n\to \infty$.

Conditioning on $A_n$ in (\ref{eq:An2}), the posterior consists of a mixture of laws conditional on $N, s$ and $\gamma$
\begin{align*}
I_n&=\E^\Pi[e^{t\sqrt n (\Psi(f)-\hat\Psi)}\mid \Y^{(n)}, A_n]\\
&=\sum_{N=1}^\infty  \Pi(N\mid \Y^{(n)}, A_n) \sum_{s=1}^T\Pi(s\mid \Y^{(n)}, A_n,N)  \sum_{\gamma \in  \mV^{\p,s}} \Pi(\gamma\mid \Y^{(n)}, A_n,N,s) I_{n,s,\gamma}\\
&=\sum_{N=1}^{N_n} \Pi(N\mid \Y^{(n)}, A_n)\sum_{s=1}^{s_n} \pi(s\mid \Y^{(n)}, A_n, N)  \sum_{\gamma \in  \mV^{\p,s}} \Pi(\gamma\mid \Y^{(n)}, A_n, N,s)  I_{n,s,\gamma} +o_p(1) 
\end{align*}
where we denote with
\[
I_{n,s,\gamma}=\E^\Pi[e^{t\sqrt n (\Psi(f)-\hat\Psi)}\mid \Y^{(n)},A_n,N,s,\gamma].
\]
The  second equality follows from the fact that $\Pi(N>N_n\mid \Y^{(n)})\to 0$ and $\Pi(s>s_n\mid \Y^{(n)})\to 0$ in $\P_0^n$ probability as $n\to \infty$, using  Corollary 6.1 of \citet{polson2018posterior}. Thereby the set $A_n$ eventually excludes all the deep learning mappings outside the sieve.

\paragraph{Linear functionals} For $\Psi(f)=\expect{a, f}_L$, when $a(\cdot)$ is a constant function, following the same strategy as in the proof of Theorem \ref{thm:linear}, we have
\[
I_{n,s,\gamma}=e^{t^2\norm{a}_L^2/2}(1+o(1))
\]
and thereby the BvM holds.

\paragraph{Squared $L^2$-norm functionals} For $\Psi(f)=\norm{f}_2^2$, we use the same strategy as in the proof of  Theorem \ref{thm:L2}. For $\alpha \in (\frac{p}{2}, p)$, we have 
\begin{equation}\label{eq:norm}
\norm{f_{0[Z]}^{N,s,\gamma}-f_0}_L^2\leq\norm{f-f_0}_L^2 =o\left(\frac{1}{\sqrt n}\right) 
\end{equation}
here $f_{0[Z]}^{N,s,\gamma}$ denotes the projection of $f_0$ onto deep learning networks with a fixed sparsity and hidden structure $(\gamma, Z)$ where $\abs{\gamma}=s$ and the width equals $N$ (similarly as in (\ref{eq:f_proj})).
The inequality \eqref{eq:norm} holds for all $f$  with a deep structure determined by $(\gamma,Z)$.

The following arguments are similar to the proof of Theorem \ref{thm:L2} but will be conditional on $N$ and $s$. Since
\begin{align*}
&\Pi(\{W_l, b_l\}_{l=1}^{L+1}\mid \Y^{(n)}, A_n,N,s,\gamma ) 
= \Pi(W_{L+1}, b_{L+1}\mid \Y^{(n)}, A_n,N,s,\gamma, Z) d\Pi(Z \mid \Y^{(n)} , A_n, N, s, \gamma)
\end{align*}
we can rewrite $I_{n,s,\gamma}$ as
\begin{align*}
I_{n,s,\gamma} &=\int \left(\int e^{t\sqrt n (\Psi(f)-\hat\Psi)} d\Pi(W_{L+1}, b_{L+1}\mid \Y^{(n)}, A_n,N,s,\gamma, Z) \right) d\Pi(Z \mid \Y^{(n)} , A_n, N, s, \gamma)\\
&=(1+o(1))e^{2t^2\norm{f_0}_L^2}\int  e^{t^2(V^{N,s,\gamma}_Z-V_0)/2+t\sqrt n(\hat \Psi^{N,s,\gamma}_Z-\hat\Psi)} d\Pi(Z \mid \Y^{(n)} , A_n, N, s, \gamma)
\end{align*}
where
\[
\hat \Psi^{N,s,\gamma}_Z=\Psi(f_0)+\frac{1}{\sqrt n}W_n(2f_{0[Z]}^{N,s,\gamma}),  \quad V^{N,s,\gamma}_Z=4\norm{f_{0[Z]}^{N,s,\gamma}}_L^2.
\]
and the term $(1+o(1))$ comes from similar considerations as in (\ref{eq:quad_local_In}).

Now we need to show $I_{n,s,\gamma}\to e^{2t^2\norm{f_0}_L^2}$ for all $N, s$ and  $\gamma$ in the local neighborhood $A_n$. In other words,
\begin{align}
\displaystyle
\sup_{N \leq N_n}\sup_{s \leq s_n} \sup_{\gamma \in \mV^{\p,s}} \int e^{t^2(V^{N,s,\gamma}_Z-V_0)/2+t\sqrt n(\hat \Psi^{N,s,\gamma}_Z-\hat\Psi)} d\Pi(Z \mid \Y^{(n)} , A_n,N,s,\gamma)
=o_P(1).\label{eq:sup_bound1}
\end{align}

Then we can write for $\alpha >p/2$
\begin{align*} 
\sqrt n (\hat\Psi^{N,s,\gamma} -\hat \Psi)&=W_n(f_{0[Z]}^{N,s,\gamma}-f_0)=o_P(1),\\
\abs{V_{N,s,\gamma}-V_0}&=4\abs{\norm{f_{0[Z]}^{N,s,\gamma}}_L^2-\norm{f_0}_L^2} \\
&\lesssim  2\norm{f_0}_L\norm{f_{0[Z]}^{N,s,\gamma}-f_0}_L+ \norm{f_{0[Z]}^{N,s,\gamma}-f_0}_L^2\\
&\lesssim \norm{f_{0[Z]}^{N,s,\gamma}-f_0}_L\leq \xi_n.
\end{align*}
With $\alpha >p/2$,   (\ref{eq:sup_bound1}) is satisfied. Aggregating the sum of $I_{N,s,\gamma}$ over $N,s$ and $\gamma$, we have
\begin{align*}
I_n &=\sum_{N=1}^{N_n} \Pi(N\mid \Y^{(n)}, A_n)\sum_{s=1}^{s_n} \Pi(s\mid \Y^{(n)}, A_n, N)  \sum_{\gamma \in  \mV^{\p,s}} \Pi(\gamma\mid \Y^{(n)}, A_n, N,s)  I_{n,s,\gamma} +o_P(1)\\
&= \sum_{N=1}^{N_n} \Pi(N\mid \Y^{(n)}, A_n)\sum_{s=1}^{s_n} \Pi(s\mid \Y^{(n)}, A_n, N)  \sum_{\gamma \in  \mV^{\p,s}} \Pi(\gamma\mid \Y^{(n)}, A_n, N, s) (1+o(1)) e^{2t^2 \norm{f_0}_L^2 +o_P(1)}+o_P(1).
\end{align*}
As a result, we have $I_n\rightarrow e^{2t^2 \norm{f_0}_L^2}$ for all $t\in \R$ as $n\rightarrow\infty$, which concludes the proof  for the $L^2$-norm functional case. \qedhere

\end{document}